\documentclass[reqno]{amsart}

\usepackage[dvipsnames]{xcolor}

\usepackage{iftex}
\ifPDFTeX 
  \usepackage[utf8]{inputenc}
  \usepackage[T1]{fontenc}
\else 
  \usepackage{fontspec}
\fi
\usepackage{lmodern}

\usepackage{geometry}
\usepackage[textwidth=28mm]{todonotes}
\usepackage{mathtools}
\usepackage{yfonts}
\usepackage[foot]{amsaddr}

\newcommand{\myParagraph}[1]{\medskip\noindent \textbf{#1.}}

\usepackage[pdfencoding=auto]{hyperref}
\hypersetup{
    colorlinks,
    citecolor=green,
    filecolor=black,
    linkcolor=blue,
    urlcolor=black
}

\newtheorem{Proposition}{Proposition}
\newtheorem{Theorem}{Theorem}
\newtheorem{Lemma}{Lemma}
\newtheorem{Remark}{Remark}

\newtheorem{Corollary}{Corollary}

\newcommand{\dt}{\,\partial_t\, }
\newcommand{\dx}{\,\partial_x\, }

\newcommand{\ddt}{\frac{\dd}{\dd t}}
\newcommand{\dxx}{\,\partial_{xx}\, }
\newcommand{\dtx}{\,\partial_{tx}\, }

\newcommand{\dd}{\,\mathrm{d}}

\newcommand{\R}{\mathbb{R}}
\newcommand{\TT}{\mathbb{T}}

\newcommand{\LL}{\mathsf{L}}

\newcommand{\II}{\mathcal{I}}
\newcommand{\J}{\mathcal{J}}

\newcommand{\fg}{\mathfrak{f}}
\newcommand{\ggot}{\mathfrak{g}}
\newcommand{\Fg}{\mathfrak{F}}
\newcommand{\fgt}{\tilde{\mathfrak{f}}}
\newcommand{\ggt}{\tilde{\mathfrak{g}}}
\newcommand{\hg}{\mathfrak{h}}
\newcommand{\hgt}{\tilde{\mathfrak{h}}}
\newcommand{\Fgt}{\widetilde{\mathfrak{F}}}

\newcommand{\iph}{{i+\frac{1}{2}}}
\newcommand{\imh}{{i-\frac{1}{2}}}
\newcommand{\jph}{{j+\frac{1}{2}}}
\newcommand{\jmh}{{j-\frac{1}{2}}}

\newcommand{\inv}{^{-1}}

\newcommand{\rhoinf}{\rho_\infty}
\newcommand{\rhoinfs}{\rho_\infty^*}

\newcommand{\overbar}[1]{\mkern 1.5mu\overline{\mkern-1.5mu#1\mkern-1.5mu}\mkern 1.5mu}

\numberwithin{equation}{section}

\title{Discrete hypocoercivity for a nonlinear kinetic reaction model}
\author{Marianne Bessemoulin-Chatard}
\address[Marianne Bessemoulin-Chatard]{Nantes Université, CNRS, Laboratoire de Mathématiques Jean Leray, LMJL,
UMR 6629, F-44000 Nantes, France}
\author{Tino Laidin}
\address[Tino Laidin]{Univ. Lille, CNRS, Inria, UMR 8524 - Laboratoire Paul Painlevé, F-59000 Lille, France}
\author{Thomas Rey}
\address[Thomas Rey]{Université Côte d’Azur, CNRS, LJAD, Parc Valrose, F-06108 Nice, France}
\email{marianne.bessemoulin@univ-nantes.fr}
\email{tino.laidin@univ-lille.fr}
\email{thomas.rey@univ-cotedazur.fr}

\begin{document}

\begin{abstract}
In this article, we propose a finite volume discretization of a one dimensional nonlinear reaction kinetic model proposed in \cite{NeumannSchmeiser2016}, which describes a 2-species recombination-generation process. Specifically, we establish the long-time convergence of approximate solutions towards equilibrium, at exponential rate. The study is based on an adaptation for a discretization of the linearized problem of the $L^2$ hypocoercivity method introduced in \cite{DolbeaultMouhotSchmeiser2015}. From this, we can deduce a local result for the discrete nonlinear problem. As in the continuous framework, this result requires the establishment of a maximum principle, which necessitates the use of monotone numerical fluxes. 

\end{abstract}


\maketitle


\section{Introduction}

\subsection{On a kinetic generation-recombination model}

In this article, we construct and analyze the decay to equilibrium of a finite volume scheme for a one-dimensional kinetic relaxation model describing a generation-recombination reaction of two species proposed in \cite{NeumannSchmeiser2016}, which can be seen as a simplified version of models describing generation and recombination of electron-hole pairs in semiconductors (see \textit{e.g.}~\cite{DegondNouriSchmeiser2000}). More precisely, we consider the system
\begin{align}
    &\dt \fg + v\dx \fg = \chi_1 - \rho_\ggot \fg, \label{eq_f_nonlin}\\
    &\dt \ggot + v\dx \ggot = \chi_2 - \rho_\fg \ggot, \label{eq_g_nonlin}
\end{align}
where $\fg$ and $\ggot$ represent the phase space densities of two chemical reactants A and B. These reactants are produced by the decomposition of a substance C (whose density is assumed to be fixed) with nonnegative velocity profiles $\chi_1$ and $\chi_2$, and can also recombine to form C and thus be eliminated from the system. The densities $\fg$ and $\ggot$ depend on time $t\geq 0$, position $x \in \TT$ the one-dimensional torus, and velocity $v\in\R$. The probability of the reaction depends on the position density of the reaction partner, defined by
\begin{equation}\label{def_densities}
    \rho_{\hg}(t,x)\coloneqq\int_\R \hg(t,x,v)\,\dd v, \qquad \hg=\fg,\,\ggot.
\end{equation}
The system \eqref{eq_f_nonlin}--\eqref{eq_g_nonlin} is completed by the initial condition
\begin{equation}\label{CI}
    \fg(0,x,v)=\fg_I(x,v)\geq 0,\quad \ggot(0,x,v)=\ggot_I(x,v)\geq 0.
\end{equation}
Since the chemical reaction under consideration is assumed to be reversible, it is required that $\int_\R (\chi_1 -\chi_2)\,\dd v=0$. We moreover make the following assumptions about the moments of the velocity profiles, which are given positive functions of $v$: for $k=1,\,2$,
\begin{equation}\label{hyp_chi}
    \begin{gathered}
        \int_\R \chi_k\,\dd v=1,\qquad \int_\R v\,\chi_k\,\dd v=0,\\
        D_k\coloneqq\int_\R v^2\,\chi_k\,\dd v<\infty,\quad Q_k\coloneqq\int_\R v^4\,\chi_k\,\dd v<\infty.
    \end{gathered}
\end{equation}

The purpose of this article is to propose a numerical scheme for the system \eqref{eq_f_nonlin}--\eqref{eq_g_nonlin} for which we are able to rigorously study the long-time behavior. More precisely, remarking that the mass difference is conserved:
\begin{equation*}
    \frac{\dd}{\dd t}\int_{\TT\times\R}(\fg-\ggot)\,\dd v\,\dd x=0,
\end{equation*}
let us introduce the unique constant $\rho_\infty>0$ such that
\begin{equation}\label{def_rhoinf}
    \int_{\TT\times\R}(\fg_I-\ggot_I)\,\dd v\,\dd x=|\TT|\left(\rho_\infty-\frac{1}{\rho_\infty}\right).
\end{equation}
The equilibrium state $\Fg_\infty=(\fg_\infty,\ggot_\infty)$, depending only on velocity, is then defined by
\begin{equation}\label{def_eq}
    \fg_\infty(x,v)\coloneqq\rho_\infty\,\chi_1(v),\qquad \ggot_\infty(x,v)\coloneqq\frac{1}{\rho_\infty}\chi_2(v).
\end{equation}

In \cite{NeumannSchmeiser2016}, the exponential decay to equilibrium (what we shall call in this paper hypocoercivity) of the solution to the linearization of \eqref{eq_f_nonlin}--\eqref{eq_g_nonlin} is established by applying the method proposed in \cite{DolbeaultMouhotSchmeiser2015}. Then, this result is extended to a local result for the nonlinear case, thanks to the proof of maximum principle estimates for the nonlinear problem.

\subsection{Hypocoercivity}

Due to the relaxation structure of the right hand side of system \eqref{eq_f_nonlin}--\eqref{eq_g_nonlin} and to the mixing properties of the free transport operator on the torus, one can naturally expect that the solutions to this type of problem  will exhibit a fast relaxation towards $\Fg_\infty$.
This is a classical problem, which can be traced back to the seminal work of Hörmander \cite{Hoermander1967} on the hypoellipticity of linear operators. 
    
The first partial proof of this large time behavior can be found in \cite{Bensoussan1979}.
    It was then proved for a very large class of collision operators (with or without confinement in velocity) in  \cite{HerauNier2004,Herau2017}, and in the series of papers \cite{DolbeaultMouhotSchmeiser2015,Bouin2020} (see also the references therein) that for a suitable norm, the rate of convergence towards the equilibrium is exponential: there exists some constants $\lambda >0$ and $C \geq 1$ such that
    \begin{equation*}
      \left \| \hg(t) - \hg_\infty \right \|_\mathcal{X} \leq C \left \| \hg_I - \hg_\infty \right \|_\mathcal{X}  e^{-\lambda t},
    \end{equation*}
    in a well chosen Hilbert space $\mathcal{X}$. We shall call this type of behavior \emph{hypocoercivity} \cite{VillaniHypo2009} (the case $C=1$ corresponding to classical coercive behavior).
     
    A very {robust proof} for establishing such a result can be found in the seminal work \cite{DolbeaultMouhotSchmeiser2015}. It introduces a so-called ``modified entropy functional'' that  is equivalent to a weighted $L^2$ norm and decays exponentially fast toward the global Maxwellian equilibrium.
    This technique was then used extensively in many applications, showing its robustness and practicality. One can cite for example its use on the kinetic Keller-Segel-type chemotaxis models (the so-called Othmer, Dunbar and Alt model) in \cite{CalvezRaoulSchmeiser2015}; on the study of the fractionnal Fokker-Planck equation in \cite{BouinDolbeaultLafleche2022}; on the fine properties of a large class of Vlasov-Fokker-Planck models in \cite{BouinDolbeaultZiviani2023}; on the large time behavior of the Vlasov-Poisson-Fokker-Planck equation in \cite{Addala2021}.
    Note that this technique has been very recently refined to singular problems, with motivation in control theory in \cite{DietertHerauHutridurgaMouhot2022}.

    \smallskip
    
    The goal of the current work is to establish hypocoercive properties of finite volume types discretizations of the nonlinear system \eqref{eq_f_nonlin}--\eqref{eq_g_nonlin}. Such an objective has been classical in the last decade for numerical approximations of macroscopic models exhibiting an entropic structure, see for example the classical works \cite{BurgerCarrilloWolfram2010,GosseToscani2006,ChainaisFilbet2007,BessemoulinChatardFilbet2012}.  A particular emphasis on the accurate discretization of steady states was done in \cite{Filbet2015,PareschiRey2017}. These results were then extended to problems with nonhomogeneous boundary conditions using the so-called $\phi$-entropy method in \cite{FilbetHerda2017,ChainaisHillairetHerda2020}. 
    
    Extension to numerical discretizations of  kinetic models which are able to reproduce accurately this type of behavior then becomes natural. This can be considered as a special case of Asymptotic Preserving (AP) schemes, where the asymptotic regime preserved is the large time behavior of the model. 
    The first works on the subject considered the discretization of the so-called Kolmogorov master equation, which is a simplified kinetic model where the collision operator is only a Laplacian in velocity. Both finite differences and finite elements approaches were considered in the papers \cite{PorrettaZuazua2017, Georgoulis2018}. Such AP schemes were then recently used to stabilize a finite element solver in \cite{dong2024hypocoercivityexploiting}.
    The first numerical work ``à la Villani'' about the $H^1$-hypocoercivity of the linear kinetic Fokker-Planck equation using a finite difference approach was then published in \cite{DujardinHerauLaffitte2018}. It uses the framework introduced in \cite{VillaniHypo2009} to establish the hypocoercive properties of the numerical scheme. 
    The $L^2$ method from \cite{DolbeaultMouhotSchmeiser2015} was then used to establish both hypocoercivity properties and uniform stability of a finite volume scheme for both the linear BGK and kinetic Fokker-Planck equations in \cite{BessemoulinHerdaRey2020}. This is the approach we will generalize in the current work.
    Let us also mention the recent similar work about the hypocoercive properties of a finite volume method for the fractional Fokker-Planck equation in \cite{AyiHerdaHivertTristani2023}.
    Finally, extensions of the $L^2$ method to the Vlasov equation (both with an external electric field or the linearized Vlasov-Poisson system), using an Hermite expansion  in the velocity space was recently developed in \cite{BlausteinFilbet2024_1,BlausteinFilbet2024_2}.



\smallskip
The outline of this article is as follows. In Section \ref{sec:continuous}, we briefly present the results in the continuous setting, since we are going to adapt them in the discrete framework in the following. In Section \ref{sec:discrete_setting}, we introduce the discrete setting, in particular the definition of the numerical schemes for the nonlinear problem \eqref{eq_f_nonlin}--\eqref{eq_g_nonlin} and for its linearization around the equilibrium. Section \ref{sec:hypoco_lin} is devoted to the adaptation of the $L^2$-hypocoercivity method of \cite{DolbeaultMouhotSchmeiser2015} for the discretization of the linearized problem. Then in Section \ref{sec:discrete_nonlin}, we establish the discrete counterpart of the local result for the nonlinear case. As in the continuous framework, it requires maximum principle estimates, which necessitates the use of monotone numerical fluxes. This motivates our choice of the Lax-Friedrichs fluxes for the transport terms, unlike the centered fluxes used in \cite{BessemoulinHerdaRey2020}. Finally, in Section \ref{sec:numeric}, we present some numerical experiments to illustrate the obtained theoretical results.



\section{The continuous setting}\label{sec:continuous}

In this section, we recall the main lines of the result in the continuous framework \cite{NeumannSchmeiser2016}. The decay towards equilibrium will be estimated quantitatively in the following weighted $L^2$ space:
\begin{equation}\label{def_espaceL2}
    \mathcal{H}\coloneqq L^2(\TT \times\R,\dd x\,\dd v/\chi_1)\times L^2(\TT\times \R,\dd x\,\dd v/\chi_2),
\end{equation}
endowed with the scalar product weighted with the steady state measure:
 \begin{equation*}
        \left\langle F_1, F_2\right\rangle=\int_{\TT\times\R} \left(\frac{f_1\,f_2}{\rho_\infty\,\chi_1} + \frac{g_1\,g_2\,\rho_\infty}{\chi_2}\right) \dd v\dd x, \quad \text{for }F_k=\begin{pmatrix}f_k\\g_k\end{pmatrix}.
    \end{equation*}
The corresponding norm is denoted by $\|\cdot\|$.

\subsection{The linearized case}

Let us now introduce the linearization of the system \eqref{eq_f_nonlin}--\eqref{eq_g_nonlin} around the equilibrium $\Fg_\infty$. We shall denote by  $(f,g)$ the perturbations $\fg=\fg_\infty+\varepsilon\,f$, $\ggot=\ggot_\infty+\varepsilon\,g$. Formally, in the limit $\varepsilon \to 0$, one gets the following linearized problem:
\begin{align}
        &\dt f + v\dx f = -\rhoinf\chi_1\rho_g - \rhoinf\inv f, \label{eq_f_lin}\\
        &\dt g + v\dx g = -\rhoinf\inv\chi_2\rho_f-\rhoinf g.\label{eq_g_lin}
\end{align}
Remark that the perturbations $f$ and $g$ now satisfy
\[ \int_{\TT\times\R}(f-g)\,\dd x\,\dd v=\int_\TT (\rho_f-\rho_g)\,\dd x = 0.\]
The orthogonal projection onto the null space of the linearized collision operator
\begin{equation*}
        \LL F=\begin{pmatrix}
            -\rhoinf\chi_1\rho_g - \rhoinf\inv f \\
            -\rhoinf\inv\chi_2\rho_f-\rhoinf g
        \end{pmatrix},\quad \text{for }F=\begin{pmatrix}
            f\\ g
        \end{pmatrix},
    \end{equation*}
is given by
\begin{equation*}
        \Pi F = \frac{\rho_f-\rho_g}{\rhoinf^2+1}\begin{pmatrix}
            \rhoinf^2\chi_1\\ -\chi_2
        \end{pmatrix}.
\end{equation*}
First of all, the following microscopic coercivity estimate can be established.
\begin{Lemma}[Microscopic coercivity]\label{lem_coer_micro_continu}
Let \eqref{hyp_chi} hold and let $F=(f,g)$ be the solution to the linearized system \eqref{eq_f_lin}--\eqref{eq_g_lin} with initial data $F_I=(f_I,g_I)\in\mathcal{H}$ satisfying $\int_{\TT\times\R}(f_I-g_I)\,\dd x\,\dd v=0$.

Then for every $t\geq 0$,
\begin{equation*}
    \frac{1}{2}\frac{\dd}{\dd t}\|F(t)\|^2+C_{mc}\|(I-\Pi)F(t)\|^2\leq 0,on \quad C_{mc}=\min(\rhoinf,\rhoinf\inv).
\end{equation*}
\end{Lemma}

\begin{proof}
    Since $F$ is the solution to \eqref{eq_f_lin}--\eqref{eq_g_lin}, 
    \[\frac{1}{2}\frac{\dd}{\dd t}\|F(t)\|^2=\langle \LL F,F\rangle.\]
    Then straightforward computations using only the fact that $\int_\R \chi_k\,\dd v=1$ (see \cite{NeumannSchmeiser2016}) show that 
    \begin{equation}\label{eq:LFscalF_cont}
        -\langle \LL F,F\rangle \geq \min(\rhoinf,\rhoinf\inv)\,\|(I-\Pi)F\|^2,
    \end{equation}
    which concludes the proof.
\end{proof}

Let us now introduce the moments of $h\coloneqq f-g$:
\begin{equation}\label{def_moments_continus}
    u_h\coloneqq\int_\R h\,\dd v,\quad J_h\coloneqq\int_\R v\,h\,\dd v,\quad S_h\coloneqq\int_\R (v^2-D_0)\,h\,\dd v,
\end{equation}
where 
\begin{equation}\label{def_D0}
    D_0=\frac{\rho_\infty^2D_1+D_2}{\rho_\infty^2+1},
\end{equation}
$D_k$ being defined in \eqref{hyp_chi}.

By subtracting \eqref{eq_g_lin} from \eqref{eq_f_lin}, multiplying by $(1,v)$ and integrating with respect to $v$, one obtains the following moments equations for $h$:
\begin{align}
    & \dt u_h + \dx J_h = 0, \label{eq_uh}\\
    & \dt J_h + \dx S_h + D_0\dx u_h = -(\rhoinf\inv J_f-\rhoinf J_g), \label{eq_Jh}
\end{align}
where $J_f$ and $J_g$ are the first order moments of $f$ and $g$ respectively.

\begin{Lemma}[Moments estimates]\label{lem_moments_cont}
    Under assumptions \eqref{hyp_chi}, the moments $u_h$, $J_h$ and $S_h$ satisfy the following estimates:
    \begin{gather}
        \|u_h\|_{L^2(\TT)}=C_u\|\Pi F\|, \label{estim_uh_cont} \\
        \|J_h\|_{L^2(\TT)}\leq C_{J1}\|F\|, \label{estim_Jh1_cont}\\
        \|J_h\|_{L^2(\TT)}\leq C_{J1}\|(I-\Pi)F\|,\label{estim_Jh2_cont}\\
        \|S_h\|_{L^2(\TT)}\leq C_S\|(I-\Pi)F\|, \label{estim_Sh_cont}\\
        \|\rhoinf\inv J_f-\rhoinf J_g\|_{L^2(\TT)}\leq C_{J2}\|(I-\Pi)F\|, \label{estim_JfJg_cont}
    \end{gather}
    where the constants are given by:
    \begin{align}
        C_u &= \sqrt{\frac{\rhoinf^2+1}{\rhoinf}},\\
        C_{J1} &= \sqrt{2\max(\rhoinf D_1,\rhoinf\inv D_2)},\\
        C_{S} &= \sqrt{2\max(\rhoinf(Q_1-2D_0D_1+D_0^2),\rhoinf\inv(Q_2-2D_0D_2+D_0^2))},\\
        C_{J2} &= \max(\rhoinf\inv,\rhoinf)C_{J1}.
    \end{align}
\end{Lemma}

\begin{proof}
Let us start with equality \eqref{estim_uh_cont}. By definition of $\|\Pi F\|$ and using the first property in \eqref{hyp_chi} one has:
\begin{align*}
    \|\Pi F\|^2 &= \frac{1}{(\rhoinf^2+1)^{2}} \int_\TT\int_\R (\rho_f-\rho_g)^2 \left(\frac{\rhoinf^4\,\chi_1^2}{\chi_1\rhoinf} + \frac{\chi_2^2\,\rhoinf}{\chi_2}\right) \dd v\dd x\\
    &= \frac{\rhoinf}{(\rhoinf^2+1)^{2}}\int_\TT (\rho_f-\rho_g)^2(\rhoinf^2+1)\dd x\\
    &= \frac{\rhoinf}{\rhoinf^2+1}\|u_h\|_{L^2(\TT)}^2.
\end{align*}
Now for \eqref{estim_Jh1_cont}, by definition of $J_h$ and using the identity $(a-b)^2\leq2(a^2+b^2)$ one has
\begin{equation}\label{estim_Jh_Comp}
    \|J_h\|_{L^2(\TT)}^2 \leq 2\int_\TT \left(\int_\R  v\,f\frac{\sqrt{\rhoinf\chi_1}}{\sqrt{\rhoinf\chi_1}}\dd v\right)^2\dd x+2\int_\TT \left(\int_\R v\,g\sqrt{\frac{\rhoinf}{\chi_2}}\sqrt{\frac{\chi_2}{\rhoinf}}\dd v \right)^2\dd x.
\end{equation}
This yields thanks to Cauchy-Schwarz inequality
\begin{equation*}
    \|J_h\|_{L^2(\TT)}^2 \leq 2\int_\TT\left(\rhoinf\,D_1\int_\R \frac{f^2}{\rhoinf\,\chi_1}\,\dd v+\frac{D_2}{\rhoinf}\int_\R\frac{g^2\,\rhoinf}{\chi_2}\,\dd v\right)\dd x,
\end{equation*}
from which we deduce \eqref{estim_Jh1_cont}.

The second estimate \eqref{estim_Jh2_cont} on $\|J_h\|_{L^2(\TT)}$ is obtained by observing that using the second property in \eqref{hyp_chi}, one has
\begin{equation*}
    \|J_h\|_{L^2(\TT)}^2\leq 2\int_\TT \left[\left(\int_\R v\left(f-\frac{u_h}{\rhoinf^2+1}\rhoinf^2\,\chi_1\right)\dd v\right)^2+\left(\int_\R v\left(g+\frac{u_h}{\rhoinf^2+1}\,\chi_2\right)\dd v\right)^2\right]\dd x,
\end{equation*}
from which we deduce using Cauchy-Schwarz inequality
\begin{align*}
    \|J_h\|_{L^2(\TT)}^2\leq & \,2\int_\TT \rhoinf\,D_1\int_\R\left(f-\frac{u_h}{\rhoinf^2+1}\,\rhoinf^2\,\chi_1\right)^2\frac{1}{\rhoinf\,\chi_1}\dd v\dd x  \\
    & +2\int_\TT\frac{D_2}{\rhoinf}\int_\R \left(g+\frac{u_h}{\rhoinf^2+1}\,\chi_2
\right)^2\frac{\rhoinf}{\chi_2}\dd v\dd x,
\end{align*}
from which we obtain \eqref{estim_Jh2_cont}.

Regarding the estimation \eqref{estim_Sh_cont} of $\|S_h\|_{L^2(\TT)}$, we notice that by definition \eqref{def_D0} of $D_0$ 
\begin{equation*}
    \int_\R (v^2-D_0)\frac{u_h}{\rhoinf^2+1}(\rhoinf^2\chi_1+\chi_2)\dd v = 0,
\end{equation*}
and as previously, we use that $(a-b)^2\leq2(a^2+b^2)$, leading to the estimate
\begin{align*}
    \|S_h\|^2_{L^2(\TT)} \leq 2\int_\TT &\left(\int_\R (v^2-D_0)\left(f-\frac{u_h}{\rhoinf^2+1}\rhoinf^2\chi_1\right)\dd v\right)^2\dd x\\
    +2\int_\TT &\left(\int_\R (v^2-D_0)\left(g+\frac{u_h}{\rhoinf^2+1}\chi_2\right)\dd v \right)^2\dd x.\\
\end{align*}
The right hand side can then be rewritten
\begin{align*}
    \|S_h\|^2_{L^2(\TT)} \leq2\int_\TT &\left(\int_\R (v^2-D_0)\frac{\sqrt{\chi_1\rhoinf}}{\sqrt{\chi_1\rhoinf}}\left(f-\frac{u_h}{\rhoinf^2+1}\rhoinf^2\chi_1\right)\dd v\right)^2\dd x\\
    &+2\int_\TT \left(\int_\R (v^2-D_0)\sqrt{\frac{\rhoinf}{\chi_2}}\sqrt{\frac{\chi_2}{\rhoinf}}\left(g+\frac{u_h}{\rhoinf^2+1}\chi_2\right)\dd v \right)^2\dd x,
\end{align*}
and \eqref{estim_Sh_cont} is obtained after applying Cauchy-Schwarz inequality. 

Let us finally turn to the last estimate \eqref{estim_JfJg_cont}. We have
\begin{align*}
   \|\rhoinf\inv J_f-\rhoinf J_g\|^2_{L^2(\TT)} &= \int_\TT \left(\rhoinf^{-1}\int_\R vf\dd v-\rhoinf\int_\R vg\dd v \right)^2\dd x\\
    &\leq 2\max\{\rhoinf^{-2},\rhoinf^2\}\int_\TT \left[\left(\int_\R vf\dd v\right)^2+\left(\int_\R vg\dd v \right)^2\right]\dd x.
\end{align*}
The right hand side is then treated as previously.
\end{proof}

Let us now state the hypocoercivity result for the linearized problem. 

\begin{Proposition}\label{prop_hypoco_lin_cont}
There are constants $C\geq 1$ and $\kappa >0$ such that for all initial data $F_I=(f_I,g_I) \in \mathcal{H}$ such that $\int_{\TT\times\R}(f_I-g_I)\,\dd x\,\dd v=0$, the solution $F=(f,g)$ of \eqref{eq_f_lin}--\eqref{eq_g_lin} satisfies
\begin{equation*}
    \|F(t)\|\leq C\,\|F_I\|\,e^{-\kappa\,t}.
\end{equation*}
\end{Proposition}

In order to prove this proposition, we introduce following \cite{DolbeaultMouhotSchmeiser2015,NeumannSchmeiser2016} a modified entropy functional, which is as in \cite{BessemoulinHerdaRey2020} a slight simplification of the original version using the fact that we work on a bounded space domain. It reads
\begin{equation}\label{def_entropy_cont}
    H_\delta[F]\coloneqq\frac{1}{2}\|F\|^2+\delta\,\langle J_h,\dx\Phi\rangle_{L^2(\TT)},
\end{equation}
where $\Phi(t,x)$ is the solution to the Poisson equation
\begin{equation}\label{eq_aux_cont}
    -\partial_{xx}\Phi=u_h,\qquad \int_\TT \Phi\,\dd x=0,
\end{equation}
and $\delta>0$ is a small parameter to be chosen later.

\begin{Lemma}\label{lem_estim_phi_cont}
    The function $\Phi$ satisfies for all $t\geq 0$
    \begin{gather}
        \|\dx \Phi(t)\|_{L^2(\TT)}\leq C_P \, C_u\|\Pi F(t)\|, \label{estim_dxphi_cont}\\
        \|\dtx \Phi(t)\|_{L^2(\TT)}\leq C_{J1}\|(I-\Pi) F(t)\|, \label{estim_dtxphi_cont}
    \end{gather}
    where $C_P$ is the Poincaré constant of $\TT$.
\end{Lemma}

\begin{proof}
The first estimate is obtained by multiplying the auxiliary equation \eqref{eq_aux_cont} by $\Phi$, integrating on $\TT$, applying the Poincaré inequality and equality \eqref{estim_uh_cont}
\begin{equation*}
    \|\dx\Phi\|^2_{L^2(\TT)}=\langle-\partial_{xx}\Phi,\Phi\rangle_{L^2(\TT)}\leq\|u_h\|_{L^2(\TT)}\|\Phi\|_{L^2(\TT)}\leq C_uC_P\|\Pi F\|\|\dx\Phi\|_{L^2(\TT)}.
\end{equation*}
For the second estimate, we differentiate the auxiliary equation with respect to time and use the continuity equation \eqref{eq_uh} to obtain $-\partial_t\partial_{xx}\Phi=-\dx J_h$. Then we multiply by $\partial_t\Phi$ and integrate to get, thanks to \eqref{estim_Jh2_cont},
\begin{align*}
    \|\dtx \Phi(t)\|_{L^2(\TT)}^2=\langle-\dx J_h,\dt\Phi\rangle_{L^2(\TT)}&=\langle J_h,\dtx\Phi\rangle_{L^2(\TT)}\\
    &\leq C_{J1}\|(I-\Pi) F(t)\|\|\dtx \Phi(t)\|_{L^2(\TT)}.
\end{align*}
\end{proof}

Thanks to the moments estimates, for small enough $\delta >0$, the square root of the modified entropy defines an equivalent norm on $\mathcal{H}$, as stated in the following lemma.

\begin{Lemma}[Equivalent norm]\label{lem_norm_eq_cont}
There is $\delta_1>0$ such that for all $\delta \in (0,\delta_1)$, there are positive constants $0<c_\delta<C_\delta$ such that for all $F\in\mathcal{H}$, one has
\begin{equation*}
    c_\delta\|F\|^2\leq H_\delta[F]\leq C_\delta\|F\|^2.
\end{equation*}
\end{Lemma}
\begin{proof}
The result follows from the definition of the modified entropy \eqref{def_entropy_cont} and the estimate
\begin{equation*}
    |\langle J_h,\dx\Phi\rangle_{L^2(\TT)}|\leq\|J_h\|_{L^2(\TT)}\|\dx \Phi\|_{L^2(\TT)}\leq C_{J1}\,C_P\,C_u\|F\|\|\Pi F\|.
\end{equation*}
Using Young inequality and the fact that $\|\Pi F\|\leq\|F\|$,
\begin{equation*}
    |\langle J_h,\dx\Phi\rangle_{L^2(\TT)}|\leq C_{J1}\,C_u\,C_P\|F\|^2.
\end{equation*}
\end{proof}
With the previous lemmas, it is then possible to establish the proof of Proposition \ref{prop_hypoco_lin_cont}.
\begin{proof}
By using the moment equation \eqref{eq_Jh}, the time derivative of the modified entropy becomes a sum of five terms,
\begin{equation*}
    \ddt H_\delta[F]= T_1 + \delta T_2 + \delta T_3 + \delta T_4 + \delta T_5.
\end{equation*}
Using Lemmas \ref{lem_coer_micro_continu}, \ref{lem_moments_cont} and \ref{lem_estim_phi_cont}, one has
\begin{align*}
    T_1 &= \frac{1}{2}\ddt\|F\|^2\leq-C_{mc}\|(I-\Pi) F\|^2,\\
    T_2 &= -\langle\dx S_h,\dx\Phi\rangle_{L^2(\TT)}\leq\left|\langle S_h,\dxx\Phi\rangle_{L^2(\TT)}\right|\leq C_S\,C_u\|(I-\Pi) F\|\|\Pi F\|,\\
    T_3 &= -D_0\langle\dx u_h,\dx\Phi\rangle_{L^2(\TT)}=D_0\langle u_h,\dxx\Phi\rangle_{L^2(\TT)}=-D_0\,C_u^2\|\Pi F\|^2,\\
    T_4 &= -\langle\rhoinf\inv J_f -\rhoinf J_g,\dx\Phi\rangle_{L^2(\TT)}\leq C_{J2}\,C_P\,C_u\|(I-\Pi) F\|\|\Pi F\|,\\
   T_5 &= \langle J_h,\dtx\Phi\rangle_{L^2(\TT)}\leq C_{J1}^2\|(I-\Pi) F\|^2.\\
\end{align*}
Combining these estimates together, one has
\begin{align*}
    \ddt H_\delta[F] &+ (C_{mc}-\delta C_{J1}^2)\|(I-\Pi) F\|^2 + \delta D_0\,C_u^2\|\Pi F\|^2\\
    &\leq \delta\left(C_S\,C_u+C_{J2}\,C_P\,C_u\right)\|(I-\Pi) F\|\|\Pi F\|.
\end{align*}
Then, let us set $\delta\in(0,\min(\delta_1,\delta_2))$ where $\delta_1$ appears in Lemma \ref{lem_norm_eq_cont} and $\delta_2<C_{mc}/C_{J1}^2$ ensures the positivity of $(C_{mc}-\delta C_{J1}^2)$. We also set $\delta<\delta_3$ where
\begin{equation*}
    \delta_3=C_{mc}D_0\,C_u^2\left((C_S\,C_u+C_{J2}\,C_P\,C_u)^2+C_{J1}^2D_0\,C_u^2\right)\inv
\end{equation*} allows us to obtain:
\begin{equation*}
    \ddt H_\delta[F] + K_\delta\left(\|(I-\Pi) F\|^2 + \|\Pi F\|^2\right)\leq0,
\end{equation*}
where $K_\delta=\min(C_{mc}-\delta C_{J1}^2,\delta D_0\,C_u^2)/2$. One can then use that $\Pi$ is an orthogonal projector and Lemma \ref{lem_norm_eq_cont} to obtain
\begin{equation*}
    \ddt H_\delta[F] + \frac{K_\delta}{C_\delta}H_\delta[F] \leq 0.
\end{equation*}
This gives the exponential decay of $H_\delta[F]$, which allows to conclude the proof of Proposition~\ref{prop_hypoco_lin_cont} by using the lower bound in Lemma \ref{lem_norm_eq_cont} and setting $\kappa=\frac{K_\delta}{2C_\delta}$, $C=\sqrt{C_\delta/c_\delta}$.
\end{proof}

\subsection{Extension to the nonlinear setting}

From this exponential decay result for the linearized system \eqref{eq_f_lin}--\eqref{eq_g_lin}, a local result can be established for the nonlinear case by the same method. It is based on global existence result and maximum principle estimates as stated in \cite[Theorem 1.1]{NeumannSchmeiser2016}, and recalled in the following proposition.

\begin{Proposition}[Global existence result and maximum principle]\label{prop_maxprinciple_cont}
Under assumptions \eqref{hyp_chi}, let $\rhoinf$ be defined by \eqref{def_rhoinf}, and assume that there exist positive constants $\gamma_1<\rhoinf$ and $\gamma_2$ such that the initial datum $\Fg_I=(\fg_I,\ggot_I) \in L^\infty(\TT\times\R)$ satisfies for all $(x,v)\in\TT\times\R$
\begin{gather*}
    (\rhoinf-\gamma_1)\chi_1(v)\leq \fg_I(x,v)\leq (\rhoinf+\gamma_2)\chi_1(v),\\ (\rhoinf+\gamma_2)\inv\chi_2(v)\leq \ggot_I (x,v)\leq (\rhoinf-\gamma_1)\inv \chi_2(v).
\end{gather*}
Then the system \eqref{eq_f_nonlin}--\eqref{eq_g_nonlin} with initial condition $\Fg_I$ admits a unique global mild solution $\Fg\in \mathcal{C}([0,\infty),L^\infty(\TT\times\R))^2$ satisfying for all $(t,x,v)\in [0,\infty)\times\TT\times\R$:
\begin{gather*}
    (\rhoinf-\gamma_1)\chi_1(v)\leq \fg(t,x,v)\leq (\rhoinf+\gamma_2)\chi_1(v),\\ (\rhoinf+\gamma_2)\inv\chi_2(v)\leq \ggot (t,x,v)\leq (\rhoinf-\gamma_1)\inv \chi_2(v).
\end{gather*}
\end{Proposition}

From this stability result, exponential decay of small perturbations to equilibrium can be shown for the full nonlinear problem.

\begin{Theorem}
Let \eqref{hyp_chi} hold and let $\Fg_I$ satisfies the assumptions of Proposition \ref{prop_maxprinciple_cont} with $\gamma_1$ and $\gamma_2$ small enough. 

Then the solution $\Fg$ to \eqref{eq_f_nonlin}--\eqref{eq_g_nonlin} satisfies
\begin{equation*}
    \|\Fg(t)-\Fg_\infty\|\leq C\,\|\Fg_I-\Fg_\infty\|\,e^{-\kappa\,t},
\end{equation*}
with positive constants $C$ and $\kappa$.
\end{Theorem}

\begin{proof}
Let us first denote by $Q(\fg,\ggot)$ the nonlinear collision operator given by
\begin{equation}\label{eq_Q_cont}
    Q(\fg,\ggot) = \begin{pmatrix}\chi_{1} - \rho_{\ggot}\fg \\ \chi_{2} - \rho_{\fg}\ggot\end{pmatrix}.
\end{equation}
Let us also introduce $\widetilde{\Fg} = \Fg - \Fg_\infty$. The nonlinear system \eqref{eq_f_nonlin}-\eqref{eq_g_nonlin} can now be rewritten in term of $\widetilde{\Fg}$. Since the equilibrium $\Fg_\infty$ does not depend on the spatial nor time variables, $\widetilde{\Fg}$ satisfies
\begin{align*}
    \partial_t \widetilde{\Fg} + v\,\partial_x\widetilde{\Fg}
    = \LL\widetilde{\Fg} + \left( Q(\fg,\ggot)-\LL\widetilde{\Fg}\right).
\end{align*}
In addition, one can compute the following relation
\begin{equation}\label{eq_QFminusLFt_cont}
    Q(\fg,\ggot) - \LL\widetilde{\Fg} =\begin{pmatrix}
    -(\rho_{\ggot}-\rhoinf\inv)(\fg - \rhoinf\chi_{1})\\
    -(\rho_{\fg}-\rhoinf)(\ggot -\rhoinf\inv\chi_{2})
    \end{pmatrix}.
\end{equation}
Applying Proposition \ref{prop_maxprinciple_cont} one has the following bound 
\begin{equation*}
    \left\|Q(\fg,\ggot) - \LL\widetilde{\Fg}\right\|\leq\gamma\|\widetilde{\Fg}\|.
\end{equation*}
Consequently, the estimate on the entropy becomes
\begin{equation*}
    \frac{1}{2}\frac{\dd}{\dd t}\|\Fgt(t)\|^2+C_{mc}\|(I-\Pi)\Fgt(t)\|^2-\gamma\|\widetilde{\Fg}\|^2\leq 0.
\end{equation*}

This ensures that for $\gamma$ (related to $\gamma_1$ and $\gamma_2$) small enough, the entropy dissipation of the linearized system given in Proposition \ref{prop_hypoco_lin_cont} is enough to overcome the nonlinear dynamics.
\end{proof}

\section{The discrete setting}\label{sec:discrete_setting}

In this section, we present numerical schemes for systems \eqref{eq_f_nonlin}--\eqref{eq_g_nonlin} and \eqref{eq_f_lin}--\eqref{eq_g_lin}. The schemes are implicit in time and of finite volume type in the $(x,v)$-phase space.

\subsection{Notations}

\subsubsection{Mesh}

Since it is in practice not possible to implement a numerical scheme on an unbounded domain, we first have to restrict the velocity domain to a bounded symmetric segment $[-v^*,v^*]$. We consider a mesh of this interval composed of $2L$ control volumes arranged symmetrically around $v=0$. We denote $v_\jph$ the $2L+1$ interface points, with $j\in\J\coloneqq\{-L,\cdots,L\}$. In this way,
\[ v_{-L+\frac{1}{2}}=-v^*,\quad v_{\frac{1}{2}}=0,\quad v_\jph=-v_\jmh \qquad \forall j=0,\cdots,L.\]
For the sake of simplicity, we assume that the velocity mesh is uniform, namely that every cell $\mathcal{V}_j=(v_\jmh,v_\jph)$ has constant length $\Delta v$. Denoting $v_j$ the midpoint of the cell $\mathcal{V}_j$, we also have $v_j=-v_{-j+1}$ for all $j=1,\cdots,L$.

In space, we consider a uniform discretization of the torus $\TT$ into $N$ cells
\[\mathcal{X}_i\coloneqq(x_\imh,x_\iph), \quad i\in\II\coloneqq\mathbb{Z}/N\mathbb{Z}\]
of length $\Delta x$. Once again, we assume uniformity of the mesh for simplicity's sake, but our results generalize to a non-uniform setting. As in \cite{BessemoulinHerdaRey2020}, we have to impose that $N$ is odd in order to have, among others, a discrete Poincaré inequality on the torus.

The control volumes in phase space are defined by
\[ K_{ij}\coloneqq\mathcal{X}_i\times\mathcal{V}_j,\qquad \forall (i,j)\in\II\times \J.\]
The size of this phase space discretization is defined by $\Delta=(\Delta x,\Delta v)$. Finally, we set $\Delta t>0$ the time step, and $t^n=n\Delta t$ for all $n\geq 0$.

\subsubsection{Discrete velocity profiles} For $k=1,\,2$, we assume that we are given cell values $(\chi_{k,j})_{j\in\J}\in\mathbb{R}^\J$ such that the following assumptions are fulfilled
\begin{equation}\label{hyp_chi_dis}
    \begin{aligned}
        & \chi_{k,j}>0,\quad \chi_{k,j}=\chi_{k,-j+1} \qquad \forall j=1,\cdots, L,\\
        & \sum_{j\in\J}\Delta v\,\chi_{k,j}=1,\\
        & 0 <\,\underline{D}_k\,\leq D_k^\Delta\leq\,\overline{D}_k\,<\infty,\qquad Q_k^\Delta\leq\,\overline{Q}_k\,<\infty,
    \end{aligned}
\end{equation}
where 
\[ D_k^\Delta\coloneqq\sum_{j\in\J}\Delta v\,|v_j|^2\,\chi_{k,j},\qquad Q_k^\Delta\coloneqq\sum_{j\in\J}\Delta v\,|v_j|^4\,\chi_{k,j},\]
and $\underline{D}_k$, $\overline{D}_k$, $\overline{Q}_k$ are universal constants. Typically, we define $\chi_{k,j}=c_{\Delta v}\chi_k(v_j)$ and compute $c_{\Delta v}$ in such a way that the mass of $(\chi_{k,j})_j$ is 1. Note also that the symmetry properties imply 
\begin{equation}\label{sym_chi_dis}
    \sum_{j\in\J}\Delta v\,v_j\,\chi_{k,j}=0.
\end{equation}

\subsubsection{Discrete gradients and functional setting in space} Due to our choices of discretization, we need to define several discrete gradients in space. Given a macroscopic quantity $u=(u_i)_{i\in\II}$, we define 
\begin{itemize}
\item the discrete centered gradient $D_x^c u\in\R^\II$ given by
\[(D_x^cu)_i=\frac{u_{i+1}-u_{i-1}}{2\Delta x} \qquad \forall i\in\II,\]
\item the discrete downstream gradient $D_x^-u\in\R^\II$ given by
\[(D_x^-u)_i=\frac{u_{i}-u_{i-1}}{\Delta x} \qquad \forall i\in\II,\]
\item the discrete upstream gradient $D_x^+u\in\R^\II$ given by
\[(D_x^+u)_i=\frac{u_{i+1}-u_{i}}{\Delta x} \qquad \forall i\in\II.\]
\end{itemize}
It is straightforward to see that these discrete gradients satisfy the following properties:
\begin{equation}\label{prop_gradients_dis}
    \frac{1}{2}(D_x^-+D_x^+)=D_x^c, \qquad D_x^+D_x^-=D_x^-D_x^+.
\end{equation}

For $u_k=(u_{k,i})_{i\in\II}$, $k=1,\,2$, we define the discrete $L^2$ scalar product by
\begin{equation}\label{def_psL2_dis}
    \langle u_1,u_2\rangle_2\coloneqq\sum_{i\in\II}\Delta x\,u_{1,i}u_{2,i},
\end{equation}
and denote $\|\cdot\|_2$ the corresponding norm.

Using the definition of the discrete gradients and the periodic boundary conditions, we immediately have the following algebraic properties.

\begin{Lemma}\label{ref_IPP_discret}
    For all $u=(u_i)_{i\in\II}$, $\bar{u}=(\bar{u}_i)_{i\in\II}$, it holds
    \begin{align}
       & \langle D_x^c u,\bar{u}\rangle_2=-\langle u,D_x^c\bar{u}\rangle_2, \label{IPP_centre}\\
       & \langle D_x^+ u,\bar{u}\rangle_2=-\langle u,D_x^-\bar{u}\rangle_2 \label{IPP_decentre},\\
       & \langle (D_x^+D_x^- + D_x^-D_x^+) u,\bar{u}\rangle_2=-4\langle D_x^cu,D_x^c\bar{u}\rangle_2 \label{IPP2_decentre},\\
       & \Delta x\|D_x^c u\|_2 \leq \|u\|_2 \label{estim_dxu_disc}.
    \end{align}
\end{Lemma}

Let us finally recall the discrete Poincaré inequality on the torus (see for example \cite[Lemma 6]{BessemoulinHerdaRey2020} for a proof of this result).

\begin{Lemma}[Discrete Poincaré inequality on the torus]\label{lem_Poincare_dis}
    Assume that the number of points $N$ in the space discretization of the torus is odd. Then, there is a constant $C_P>0$ independent on $\Delta x$ such that for all $u=(u_i)_{i\in\II}$ satisfying $\sum_{i\in\II}\Delta x\,u_i=0$,
    \[\|u\|_2\leq C_P\,\|D_x^c u\|_2.\]
\end{Lemma}
Note that in the above lemma, the constant $C_P$ converges towards $\frac{1}{\pi}$ while in the continuous case the Poincaré constant equals $\frac{1}{2\pi}$. For the sake of simplicity, we use the same notation. Actually, the difference originates from the choice of the discrete gradient and we refer to the discussion on the matter in \cite{BessemoulinHerdaRey2020} for further details.

\subsection{Definition of the numerical scheme for the nonlinear system}

The numerical scheme for approximating \eqref{eq_f_nonlin}--\eqref{eq_g_nonlin} is based on a finite volume discretization in phase space and backward Euler discretization in time. The initial datum $\Fg_I=(\fg_I,\ggot_I)$ is discretized by
\[\fg_{ij}^0=\frac{1}{\Delta x\,\Delta v}\int_{K_{ij}}\fg_I(x,v)\,\dd x\,\dd v, \quad \ggot_{ij}^0=\frac{1}{\Delta x\,\Delta v}\int_{K_{ij}}\ggot_I(x,v)\,\dd x\,\dd v \qquad \forall (i,j)\in\II\times\J.\]

Then, by integrating \eqref{eq_f_nonlin}--\eqref{eq_g_nonlin} on each cell $K_{ij}$, the following numerical scheme is obtained: for all $n\geq 0$, $i\in\II$, $j\in\J$,
\begin{align}
    & \frac{\fg_{ij}^{n+1}-\fg_{ij}^n}{\Delta t}+\frac{1}{\Delta x\,\Delta v}\left(\mathcal{F}_{\iph,j}^{n+1}-\mathcal{F}_{\imh,j}^{n+1}\right)=\chi_{1,j}-\rho_{\ggot,i}^{n+1}\fg_{ij}^{n+1} ,\label{scheme_f_nonlin}\\
    & \frac{\ggot_{ij}^{n+1}-\ggot_{ij}^n}{\Delta t}+\frac{1}{\Delta x\,\Delta v}\left(\mathcal{G}_{\iph,j}^{n+1}-\mathcal{G}_{\imh,j}^{n+1}\right)=\chi_{2,j}-\rho_{\fg,i}^{n+1}\ggot_{ij}^{n+1},\label{scheme_g_nonlin}
\end{align}
with the so-called Lax-Friedrichs fluxes 
\begin{align}
    & \mathcal{F}_{\iph,j}^{n+1}=\Delta v\,\frac{v_j}{2}(\fg_{i+1,j}^{n+1}+\fg_{ij}^{n+1})-\Delta v\,\lambda (\fg_{i+1,j}^{n+1}-\fg_{ij}^{n+1}),\label{flux_F}\\
    & \mathcal{G}_{\iph,j}^{n+1}=\Delta v\,\frac{v_j}{2}(\ggot_{i+1,j}^{n+1}+\ggot_{ij}^{n+1})-\Delta v\,\lambda (\ggot_{i+1,j}^{n+1}-\ggot_{ij}^{n+1}),\label{flux_G}
\end{align}
where $\lambda=\Delta x/2\,\Delta t$ is assumed to be a fixed constant. For all $n\geq 0$ and $i\in\II$, the discrete macroscopic densities are given by
\begin{equation}\label{def_rhofg_dis}
    \rho_{\fg,i}^n=\sum_{j\in\J}\Delta v\,\fg_{ij}^n,\qquad \rho_{\ggot,i}^n=\sum_{j\in\J}\Delta v\,\ggot_{ij}^n.
\end{equation}

Remark that the scheme \eqref{scheme_f_nonlin}--\eqref{scheme_g_nonlin} clearly satisfies the discrete mass conservation of $\fg-\ggot$:
\[\sum_{(i,j)\in\II\times\J}\Delta x\,\Delta v\,(\fg_{ij}^n-\ggot_{ij}^n)=\sum_{(i,j)\in\II\times\J}\Delta x\,\Delta v\,(\fg_{ij}^0-\ggot_{ij}^0) \quad \forall n\geq 0.\]
Then, we define $\rhoinfs>0$ as the unique constant such that 
\begin{equation}\label{def_rhoinf_dis}
    M_0\coloneqq\sum_{(i,j)\in\II\times\J}\Delta x\,\Delta v\,(\fg_{ij}^0-\ggot_{ij}^0)=|\TT|\left(\rhoinfs-\frac{1}{\rhoinfs}\right).
\end{equation}
In particular we take, 
\begin{equation}\label{def_rhoinf_dis_full}
    \rhoinfs = \frac{M_0+\sqrt{M_0^2+4|\TT|}}{2|\TT|}.
\end{equation}
It is clear that $\Fg^\infty=(\fg^\infty,\ggot^\infty)$ defined by
\begin{equation}\label{def_eq_dis}
    \fg^\infty_{ij}=\rhoinfs\,\chi_{1,j},\qquad \ggot^\infty_{ij}=\frac{1}{\rhoinfs}\,\chi_{2,j} \qquad \forall (i,j)\in \II\times\J
\end{equation}
is an equilibrium for the scheme \eqref{scheme_f_nonlin}--\eqref{scheme_g_nonlin}. As in the continuous framework, the study of the exponential convergence of the approximate solutions to this discrete equilibrium is done by analyzing the discretization of the linearized problem that we now introduce.

\begin{Remark}\label{remark_fluxes}
Note that one could choose to use many different types of numerical fluxes in \eqref{scheme_f_nonlin}--\eqref{scheme_g_nonlin}, such as the classical central fluxes (as was done in \cite{BessemoulinHerdaRey2020}) given by
\begin{align}
    & \mathcal{F}_{\iph,j}^{n+1, C}=\Delta v\,\frac{v_j}{2}(\fg_{i+1,j}^{n+1}+\fg_{ij}^{n+1}),\label{flux_F_C}\\
    & \mathcal{G}_{\iph,j}^{n+1, C}=\Delta v\,\frac{v_j}{2}(\ggot_{i+1,j}^{n+1}+\ggot_{ij}^{n+1}) \label{flux_G_C},
\end{align}
or the upwind fluxes, defined for $a^+ = \max (a,0)$ and $a^- = -\min (a,0)$ by
\begin{align}
    & \mathcal{F}_{\iph,j}^{n+1, U}=\Delta v\left(v_j^+ \,\fg_{i,j}^{n+1} - v_j^- \,\fg_{i+1,j}^{n+1}\right),\label{flux_F_U}\\
    & \mathcal{G}_{\iph,j}^{n+1, U}=\Delta v \left(v_j^+ \,\ggot_{i,j}^{n+1} - v_j^-\,\ggot_{i+1,j}^{n+1}\right) \label{flux_G_U}.
\end{align}
The use of the seemingly more complicated Lax-Friedrichs fluxes stems from the fact that central fluxes are not monotone in the sense of Crandall and Majda \cite{CrandallMajda1980}. This property ensures a maximum principle of the nonlinear scheme, which we will see later is needed for our main result. Upwind fluxes do enjoy such monotonicity properties, but complexify too much the analysis. 

We will see in Section \ref{sec:numeric} that the three choices yield similar numerical results in the linear setting, where monotonicity is not mandatory for the hypocoercive behavior of our method. It will nevertheless be crucial in the nonlinear case, where exponential decay do not occur with central fluxes.
\end{Remark}

\subsection{Definition of the numerical scheme for the linearized problem}
 
 As in the continuous framework, the perturbations are again denoted by $f$ and $g$. By integrating the linearized system \eqref{eq_f_lin}--\eqref{eq_g_lin} on each cell $K_{ij}$, the following numerical scheme is obtained: for all $n\geq 0$, $i\in\II$, $j\in\J$,
\begin{align}
    & \frac{f_{ij}^{n+1}-f_{ij}^n}{\Delta t}+\frac{1}{\Delta x\,\Delta v}\left(\mathcal{F}_{\iph,j}^{n+1}-\mathcal{F}_{\imh,j}^{n+1}\right)=-\rhoinfs\,\chi_{1,j}\,\rho_{g,i}^{n+1}-(\rhoinfs)\inv\,f_{ij}^{n+1},\label{scheme_f_lin}\\
    & \frac{g_{ij}^{n+1}-g_{ij}^n}{\Delta t}+\frac{1}{\Delta x\,\Delta v}\left(\mathcal{G}_{\iph,j}^{n+1}-\mathcal{G}_{\imh,j}^{n+1}\right)=-(\rhoinfs)\inv\chi_{2,j}\rho_{f,i}^{n+1}-\rhoinfs\,g_{ij}^{n+1},\label{scheme_g_lin}
\end{align}
where the numerical fluxes are still defined by \eqref{flux_F}--\eqref{flux_G} but replacing $\fg$ by $f$ and $\ggot$ by $g$, $\lambda=\Delta x/2\,\Delta t$ is assumed to be a fixed constant, and the discrete macroscopic densities are given by \eqref{def_rhofg_dis}.

For future use, we also need to define the discrete velocity moments of $h=(h_{ij}^n)_{i\in\II,j\in\J}$ given by $h_{ij}^n=f_{ij}^n-g_{ij}^n$: for all $n\geq 0$, $i\in\II$,
\begin{equation}\label{def_moments_discrets}
    u_{h,i}^n\coloneqq\sum_{j\in\J}\Delta v\,h_{ij}^n,\quad J_{h,i}^n\coloneqq\sum_{j\in\J}\Delta v\,v_j\,h_{ij}^n, \quad S_{h,i}^n\coloneqq\sum_{j\in\J}\Delta v\,(v_j^2-D_0^\Delta)\,h_{ij}^n,
\end{equation}
with 
\begin{equation}\label{def_D0_dis}
    D_0^\Delta=\frac{(\rhoinfs)^2D_1^\Delta+D_2^\Delta}{(\rhoinfs)^2+1}.
\end{equation}
Note in particular that one obtains from \eqref{hyp_chi_dis} the existence of some universal constants $\underline{D}_0$ and $\overbar{D_0}$ such that $\underline{D}_0\leq D_0^\Delta\leq\overbar{D_0}$.


\section{Numerical hypocoercivity for the linearized problem}\label{sec:hypoco_lin}

We now adapt the study of the linearized system \eqref{eq_f_lin}--\eqref{eq_g_lin} to the discrete setting, following the hypocoercivity method proposed in \cite{DolbeaultMouhotSchmeiser2015} and briefly described in Section \ref{sec:continuous}. In order to estimate the decay towards the equilibrium, we introduce the following weighted scalar product: for microscopic quantities $F_k=\begin{pmatrix} f_{k,ij}\\g_{k,ij}\end{pmatrix}_{i\in\II,j\in\J}$, $k=1,\,2$,  
\begin{equation}\label{def_psw_dis}
   \langle F_1,F_2\rangle_\Delta\coloneqq\sum_{i\in\II}\sum_{j\in\J}\Delta x\,\Delta v\left( \frac{f_{1,ij}f_{2,ij}}{\chi_{1,j}\rhoinfs}+\frac{g_{1,ij}g_{2,ij}\rhoinfs}{\chi_{2,j}}\right).
   \end{equation}
We denote by $\|\cdot\|_\Delta$ the corresponding norm.

For $(F_{ij})_{ij}=(f_{ij},g_{ij})_{ij}$, we also define the discrete counterpart of the linear collision operator
\begin{equation}\label{def_L_discret}
    (\LL^\Delta F)_{ij}\coloneqq\begin{pmatrix}-\rhoinfs\chi_{1,j}\rho_{g,i}-(\rhoinfs)\inv\,f_{ij}\\ -(\rhoinfs)\inv \chi_{2,j}\rho_{f,i}-\rhoinfs\,g_{ij}\end{pmatrix},
\end{equation}
and the orthogonal projection onto its null space:
\begin{equation}\label{def_Pi_discret}
    (\Pi^\Delta F)_{ij}\coloneqq\frac{\rho_{f,i}-\rho_{g,i}}{(\rhoinfs)^2+1}\begin{pmatrix}(\rhoinfs)^2\chi_{1,j}\\-\chi_{2,j}\end{pmatrix}.
\end{equation}

In this section, we derive the discrete counterparts of estimates given in Lemmas \ref{lem_coer_micro_continu} and \ref{lem_moments_cont}.

\begin{Lemma}[Discrete microscopic coercivity]\label{lem_coercivite_micro_dis}
Let \eqref{hyp_chi_dis} hold and let $(f_{ij}^n,g_{ij}^n)_{n\geq 0,i\in\II,j\in\J}$ solve the scheme \eqref{scheme_f_lin}--\eqref{scheme_g_lin}. Then for all $n\geq 0$,
\begin{equation}\label{estim_coercivite_micro_dis}
    \frac{1}{2}\left(\|F^{n+1}\|_\Delta^2-\|F^n\|_\Delta^2\right)+\Delta t\,C_{mc}^*\,\|(I-\Pi^\Delta)F^{n+1}\|_\Delta^2\leq 0,
\end{equation}
where $C_{mc}^*=\min((\rhoinfs)\inv,\rhoinfs)$.
\end{Lemma}

\begin{proof}
We multiply \eqref{scheme_f_lin} by $f_{ij}^{n+1}/(\rhoinfs\chi_{1,j})$ and \eqref{scheme_g_lin} by $g_{ij}^{n+1}\rhoinfs/\chi_{2,j}$, and then sum the two resulting expressions and sum over $(i,j)\in\II\times\J$. The resulting expression is of the form $A_1+A_2=A_3$ with
\begin{align*}
    A_1 &=  \sum_{(i,j)\in\II\times\J}\Delta x\Delta v \left((f_{ij}^{n+1}-f_{ij}^{n})\frac{f_{ij}^{n+1}}{\rhoinfs\chi_{1,j}} + (g_{ij}^{n+1}-g_{ij}^{n})\frac{g_{ij}^{n+1}\rhoinfs}{\chi_{2,j}}\right)\\
    A_2 &= \sum_{(i,j)\in\II\times\J}\left(\left(\mathcal{F}_{\iph,j}^{n+1}-\mathcal{F}_{\imh,j}^{n+1}\right)\frac{f_{ij}^{n+1}}{\rhoinfs\chi_{1,j}} +
    \left(\mathcal{G}_{\iph,j}^{n+1}-\mathcal{G}_{\imh,j}^{n+1}\right)\frac{g_{ij}^{n+1}\rhoinfs}{\chi_{2,j}}\right),\\
    A_3 &= \langle\LL^\Delta F^{n+1},F^{n+1}\rangle_\Delta.
\end{align*}
Let us first deal with the transport part $A_2$. By definition of the numerical flux \eqref{flux_F}:
\begin{align*}
    &\sum_{(i,j)\in\II\times\J}\left(\mathcal{F}_{\iph,j}^{n+1}-\mathcal{F}_{\imh,j}^{n+1}\right)\frac{f_{ij}^{n+1}}{\rhoinfs\chi_{1,j}}\\
    &=\sum_{(i,j)\in\II\times\J} \frac{v_j\Delta v}{2}\left(f_{i+1,j}^{n+1}f_{ij}^{n+1} - f_{ij}^{n+1}f_{i-1,j}^{n+1}\right)\frac{1}{\rhoinfs\chi_{1,j}}\\
    &\qquad- \lambda\sum_{(i,j)\in\II\times\J}\Delta x^2\left(D_x^+D_x^-f_j^{n+1}\right)_i\frac{f_{ij}^{n+1}}{\rhoinfs\chi_{1,j}}\Delta v.
\end{align*}
Due to the spatial periodic boundary conditions, the first term on the right hand side vanishes. Then, using the properties of the discrete gradients \eqref{prop_gradients_dis} and a discrete integration by parts \eqref{IPP2_decentre} on the second term, it yields
\begin{align*}
    &\sum_{(i,j)\in\II\times\J}\left(\mathcal{F}_{\iph,j}^{n+1}-\mathcal{F}_{\imh,j}^{n+1}\right)\frac{f_{ij}^{n+1}}{\rhoinfs\chi_{1,j}}\\
    &= -\lambda\sum_{(i,j)\in\II\times\J}\frac{\Delta x^2}{2}\left(D_x^+D_x^-+D_x^-D_x^+\right)\left(f_j^{n+1}\right)_i\frac{f_{ij}^{n+1}}{\rhoinfs\chi_{1,j}}\Delta v\\
    &= 2\lambda\Delta x\sum_{(i,j)\in\II\times\J}\left(D_x^c f_{j}^{n+1}\right)_i^2\frac{\Delta x\Delta v}{\rhoinfs\chi_{1,j}}.
\end{align*}
Therefore,
\begin{equation*}
    \sum_{(i,j)\in\II\times\J}\left(\mathcal{F}_{\iph,j}^{n+1}-\mathcal{F}_{\imh,j}^{n+1}\right)\frac{f_{ij}^{n+1}}{\rhoinfs\chi_{1,j}} \geq 0.
\end{equation*}
The same computations applied to \eqref{flux_G} yield
\begin{equation*}
    \sum_{(i,j)\in\II\times\J}\left(\mathcal{G}_{\iph,j}^{n+1}-\mathcal{G}_{\imh,j}^{n+1}\right)\frac{g_{ij}^{n+1}\rhoinfs}{\chi_{2,j}} \geq 0.
\end{equation*}
Consequently, one has
\begin{equation}\label{eq_flux_dissip}
    A_2\geq0
\end{equation}
Note that the choice of the Lax-Friedrichs numerical flux leads to an inequality where in previous works, the use of a centered scheme yielded exactly 0. This was expected since the Lax-Friedrichs scheme classically introduces numerical diffusion in system. Then, using the same computations as in the continuous case, one can obtain from $A_3$ a discrete counterpart of \eqref{eq:LFscalF_cont}:
\begin{equation}\label{eq_coer_micro_disc}
    \langle\LL^\Delta F^{n+1},F^{n+1}\rangle_\Delta \leq -C_{mc}^*\|(I-\Pi^\Delta)F^{n+1}\|_\Delta^2,
\end{equation}
where we set $C_{mc}^*=\min(\rhoinfs,(\rhoinfs)\inv)$. The term $A_1$ can then be bounded from below using the relation $(a^2-b^2)/2\leq a(a-b)$, yielding
\begin{equation}\label{eq_bound_dtF_disc}
    \frac{1}{2}\left(\|F^{n+1}\|_\Delta^2-\|F^n\|_\Delta^2\right)\leq A_1.
\end{equation}
Finally, combining \eqref{eq_flux_dissip}, \eqref{eq_coer_micro_disc} and \eqref{eq_bound_dtF_disc} we obtain the desired estimate.
\begin{equation*}
    \frac{1}{2}\left(\|F^{n+1}\|_\Delta^2-\|F^n\|_\Delta^2\right)+\Delta t\,C_{mc}^*\,\|(I-\Pi^\Delta)F^{n+1}\|_\Delta^2\leq 0.
\end{equation*}
\end{proof}

From this lemma, we can deduce the uniqueness and then the existence of a solution to the scheme \eqref{scheme_f_lin}--\eqref{scheme_g_lin} since it is a finite dimensional linear system.

\begin{Corollary}
The scheme \eqref{scheme_f_lin}--\eqref{scheme_g_lin} admits a unique solution $(F_{ij}^n)_{n\geq 0,i\in\II,j\in\J}$.
\end{Corollary}

Let us now give the discrete counterpart of Lemma \ref{lem_moments_cont}, that is uniform $L^2(\TT)$ bounds on the discrete moments.

\begin{Lemma}[Discrete moments estimates]\label{lem_mom_estim_disc}
Under the assumptions of Lemma \ref{lem_coercivite_micro_dis}, the discrete moments $(u_{h,i}^n)_{i\in\II}$, $(J_{h,i}^n)_{i\in\II}$ and $(S_{h,i}^n)_{i\in\II}$ satisfy the following estimates for all $n\geq 0$:
\begin{gather}
        \|u_h^n\|_{2}=C_u^*\|\Pi^\Delta F^n\|_\Delta, \label{estim_uh_dis} \\
        \|J_h^n\|_{2}\leq C_{J1}^*\|F^n\|_\Delta, \label{estim_Jh1_dis}\\
        \|J_h^n\|_{2}\leq C_{J1}^*\|(I-\Pi^\Delta)F^n\|_\Delta,\label{estim_Jh2_dis}\\
        \|S_h^n\|_{2}\leq C_{S}^*\|(I-\Pi^\Delta)F^n\|_\Delta, \label{estim_Sh_dis}\\
        \|(\rhoinfs)\inv J_f^n-\rhoinfs J_g^n\|_{2}\leq C_{J2}^*\|(I-\Pi^\Delta)F^n\|_\Delta, \label{estim_JfJg_dis}
    \end{gather}
    where the constants are given by:
    \begin{align}
        C_u^* &= \sqrt{\frac{(\rhoinfs)^2+1}{\rhoinfs}},\\
        C_{J1}^* &= \sqrt{2\max(\rhoinfs \overbar{D_1},(\rhoinfs)\inv \overbar{D_2})},\\
        C_{S}^* &= \sqrt{2\max(\rhoinfs(\overbar{Q_1}-2\underline{D_0}\,\underline{D_1}+\overbar{D_0}^2),(\rhoinfs)\inv(\overbar{Q_2}-2\underline{D_0}\,\underline{D_2}+\overbar{D_0}^2))},\\
        C_{J2}^* &= \max((\rhoinfs)\inv,\rhoinfs)C_{J1}^*.
    \end{align}
\end{Lemma}

\begin{proof}
The strategy is very similar to the proof in continuous case. Estimate \eqref{estim_uh_dis} is directly obtained from the definition of the discrete projector $\Pi^\Delta$:
\begin{align*}
    \|\Pi^\Delta F^n\|_\Delta^2 &= \frac{1}{((\rhoinfs)^2+1)^2} \sum_{(i,j)\in\II\times\J}\left[(\rho_{f,i}^n-\rho_{g,i}^n)^2 \frac{(\rhoinfs)^4\chi_{1,j}^2}{\chi_{1,j}\rhoinfs} + (\rho_{f,i}^n-\rho_{g,i}^n)^2\frac{\chi_{2,j}^2\rhoinfs}{\chi_{2,j}}\right]\Delta v\Delta x\\
    &= \frac{\rhoinfs}{(\rhoinfs)^2+1}\|u_h^n\|_2^2.
\end{align*}
The remaining estimates rely on the discrete properties of unit mass \eqref{hyp_chi_dis} and null odd-moments \eqref{sym_chi_dis} of the discrete velocity profiles. Estimates \eqref{estim_Jh1_dis} and \eqref{estim_Jh2_dis} are obtained noticing that
\begin{equation*}
    \|J_h^n\|_{2}^2 \leq 2\sum_{i\in\II} \left[\left(\sum_{j\in\J} v_j  \frac{\sqrt{\chi_{1,j}\rhoinfs}}{\sqrt{\chi_{1,j}\rhoinfs}}f_{ij}^n\Delta v\right)^2 + \left(\sum_{j\in\J} v_j \frac{\sqrt{\chi_{2,j}}}{\sqrt{\rhoinfs}}\sqrt{\frac{\rhoinfs}{\chi_{2,j}}} g_{ij}^n\Delta v\right)^2\right]\Delta x
\end{equation*}
and
\begin{align*}
    \|J_h^n\|_{2}^2 &\leq 2\sum_{i\in\II}\left[ \left(\sum_{j\in\J} v_j\frac{\sqrt{\chi_{1,j}\rhoinfs}}{\sqrt{\chi_{1,j}\rhoinfs}} \left(f_{ij}^n-\frac{u_{h,i}^n}{(\rhoinfs)^2+1}(\rhoinfs)^2\chi_{1,j}\right)\Delta v\right)^2\right.\\
    &\qquad\left.+\left(\sum_{j\in\J} v_j\frac{\sqrt{\chi_{2,j}}}{\sqrt{\rhoinfs}}\frac{\sqrt{\rhoinfs}}{\sqrt{\chi_{2,j}}} \left(g_{ij}^n+\frac{u_{h,i}^n}{(\rhoinfs)^2+1}\chi_{2,j}\right)\Delta v\right)^2\right]\Delta x.
\end{align*}
Using definition \eqref{def_D0_dis} of $D_0^\Delta$, one has for all $i\in\II$
\begin{equation*}
    \sum_{j\in\J}(v_j^2-D_0^\Delta)\frac{u_{h,i}^n}{(\rhoinfs)^2+1}\left((\rhoinfs)^2\chi_{1,j}+\chi_{2,j}\right)\Delta v = 0.
\end{equation*}
Estimate \eqref{estim_Sh_dis} is then obtained from 
\begin{align*}
    \|S_h^n\|_{2}^2 &\leq 2\sum_{i\in\II}\left[ \left(\sum_{j\in\J} (v_j^2-D_0^\Delta)\frac{\sqrt{\chi_{1,j}\rhoinfs}}{\sqrt{\chi_{1,j}\rhoinfs}}\left(f_{ij}^n-\frac{u_{h,i}^n}{(\rhoinfs)^2+1}(\rhoinfs)^2\chi_{1,j}\right)\Delta v\right)^2\right.\\
    &\quad\left.+\left(\sum_{j\in\J} (v_j^2-D_0^\Delta)\sqrt{\frac{\chi_{2,j}}{\rhoinfs}}\sqrt{\frac{\rhoinfs}{\chi_{2,j}}}\left(g_{ij}^n+\frac{u_{h,i}^n}{(\rhoinfs)^2+1}\chi_{2,j}\right) \Delta v\right)^2\right]\Delta x.
\end{align*}
One can then apply Cauchy-Schwarz inequalities and the boundedness of the moments of the velocity profiles to obtain the result. Finally, the last estimate \eqref{estim_JfJg_dis} is obtained by following exactly the same reasoning as in the continuous case to obtain \eqref{estim_JfJg_cont}.
\end{proof}

\begin{Lemma}[Moments schemes]\label{lem_eq_moments_dis}
Under the assumptions of Lemma \ref{lem_coercivite_micro_dis}, the discrete moments satisfy the following equations: for all $i\in\II$, $n\geq 0$,
\begin{align}
    & \frac{u_{h,i}^{n+1}-u_{h,i}^n}{\Delta t}+(D_x^cJ_h^{n+1})_i-\frac{\Delta x\,\lambda}{2}\left((D_x^+D_x^-+D_x^-D_x^+)u_h^{n+1}\right)_i=0, \label{scheme_uh_lin}\\
    &  \frac{J_{h,i}^{n+1}-J_{h,i}^n}{\Delta t}+(D_x^cS_h^{n+1})_i+D_0^\Delta\,(D_x^cu_h^{n+1})_i-\frac{\Delta x\,\lambda}{2}\left((D_x^+D_x^-+D_x^-D_x^+)J_h^{n+1}\right)_i=\nonumber \\
    & \qquad -((\rhoinfs)\inv\,J_{f,i}^{n+1}-\rhoinfs\,J_{g,i}^{n+1}).\label{scheme_Jh_lin}
\end{align}
\end{Lemma}

\begin{proof}
Let us start by subtracting \eqref{scheme_g_lin} from \eqref{scheme_f_lin}. By linearity of the numerical scheme, we obtain
\begin{equation*}
    \begin{split}
    \frac{h_{ij}^{n+1}-h_{ij}^n}{\Delta t} &+\frac{1}{\Delta x\,\Delta v}\left(\mathcal{F}_{\iph,j}^{n+1}-\mathcal{F}_{\imh,j}^{n+1}-(\mathcal{G}_{\iph,j}^{n+1}-\mathcal{G}_{\imh,j}^{n+1})\right)\\
    &=\rhoinfs(g_{ij}^{n+1}-\chi_{1,j}\rho_{g,i}^{n+1})-(\rhoinfs)\inv (f_{ij}^{n+1}-\chi_{2,j}\rho_{f,i}^{n+1}).
\end{split}
\end{equation*}
Using the definition of the numerical fluxes \eqref{flux_F}--\eqref{flux_G} and the properties of the discrete gradients \eqref{prop_gradients_dis}, one gets
\begin{equation}\label{scheme_h_lin}
    \begin{split}
    \frac{h_{ij}^{n+1}-h_{ij}^n}{\Delta t} &+ v_j(D_x^c h_j^{n+1})_i - \frac{\lambda\Delta x}{2}\left((D_x^+D_x^-+D_x^-D_x^+)h_j^{n+1}\right)_i\\
    &\qquad=\rhoinfs(g_{ij}^{n+1}-\chi_{1,j}\rho_{g,i}^{n+1})-(\rhoinfs)\inv (f_{ij}^{n+1}-\chi_{2,j}\rho_{f,i}^{n+1}).
\end{split}
\end{equation}
Finally, the moment schemes are obtained by multiplying \eqref{scheme_h_lin} by $(\Delta v, v_j\Delta v)$, summing over $j\in\J$ and applying definitions \eqref{def_moments_discrets} of the discrete moments.
\end{proof}

Thanks to these three lemmas, we are now in position to establish the discrete counterpart of Proposition \ref{prop_hypoco_lin_cont}, namely the exponential decay to equilibrium for the discrete linearized problem.

\begin{Theorem}\label{theo_hypoco_dis_lin}
Assuming that \eqref{hyp_chi_dis} is fulfilled and that the number of points $N$ of the spatial discretization is odd, there exist constants $C\geq 1$ and $\kappa>0$ such that for all $\Delta t\leq \Delta t_{\max}$ and all initial data $F^0=(f_{ij}^0,g_{ij}^0)_{(i,j)\in\II\times\J}$ such that $\sum_{(i,j)\in\II\times\J}\Delta x\,\Delta v\, (f_{ij}^0-g_{ij}^0)=0$, the solution $F^n=(f_{ij}^n,g_{ij}^n)_{i\in\II,j\in\J}$ of \eqref{scheme_f_lin}--\eqref{scheme_g_lin} satisfies for all $n\geq 0$
\begin{equation}
    \|F^{n+1}\|_\Delta\leq C \, \|F^0\|_\Delta\,e^{-\kappa\,t^n}.
\end{equation}
The constants $C$ and $\kappa$ do not depend on the size of the discretization $\Delta$, and $\Delta t_{\max}$ can be arbitrarily chosen.  
\end{Theorem}

To prove this result, let us introduce the discrete modified entropy functional 
\begin{equation}\label{def_entropy_dis}
    H_\delta^\Delta[F^n]\coloneqq\frac{1}{2}\|F^n\|_\Delta^2+\delta\langle J_h^n,D_x^c\Phi^n\rangle_2+\frac{\delta}{2\,\Delta t}\sum_{i\in\II}\Delta x\left((D_x^c\Phi^n)_i-(D_x^c\Phi^{n-1})_i\right)^2,
\end{equation}
where $\delta>0$ will be determined later and $(\Phi_i^n)_{i\in\II}$ is the solution to the following discrete Poisson equation:
\begin{equation}\label{eq_aux_dis}
    (D_x^cD_x^c\Phi^n)_i=-u_{h,i}^n \quad \forall i\in\II, \qquad \sum_{i\in\II}\Delta x\,\Phi_i^n=0.
\end{equation}

For an odd number of points $N$ of the spatial discretization, existence and uniqueness of $(\Phi_i^n)_{i\in\II}$ satisfying \eqref{eq_aux_dis} is obtained (see \cite{BessemoulinHerdaRey2020}).

Let us now derive some discrete estimates on $(\Phi_i^n)_{i\in\II}$.

\begin{Lemma}\label{lem_estim_phi_dis}
Under the assumptions of Theorem \ref{theo_hypoco_dis_lin}, one has for all $n\geq 0$
\begin{align}
    & \|D_x^c\Phi^n\|_2\leq C_P\,C_u^*\,\|\Pi^\Delta F^n\|_\Delta, \label{estim_dxphi_dis}\\
    & \|D_x^c\Phi^{n+1}-D_x^c\Phi^n\|_2\leq \Delta t\,\|J_h^{n+1}\|_2+2\,\Delta t\,\lambda\,C_u^*\,\|\Pi^\Delta F^{n+1}\|_\Delta, \label{estim_dtxphi_dis}
\end{align}
where $C_P$ is the discrete Poincaré constant of Lemma \ref{lem_Poincare_dis}.
\end{Lemma}

\begin{proof}
To show \eqref{estim_dxphi_dis}, we multiply \eqref{eq_aux_dis} by $\Phi_i^n\Delta x$ and sum over $i\in\II$. Then, discrete integration by parts together with Lemmas \ref{lem_Poincare_dis} and \ref{lem_mom_estim_disc} are used to obtain:
\begin{equation*}
    \|D_x^c\Phi^n\|_2^2 = \langle -(D_x^cD_x^c\Phi^{n}),\Phi^n \rangle_2=\langle u_h^{n},\Phi^n \rangle_2\leq \|u_h^n\|_2\|\Phi_n\|_2\leq C_P C_u^*\|\Pi^\Delta F^n\|_\Delta\|D_x^c\Phi_n\|_2.
\end{equation*}

Let us now prove \eqref{estim_dtxphi_dis}. The first step is to subtract \eqref{eq_aux_dis} at time $t^n$ from \eqref{eq_aux_dis} at time $t^{n+1}$, multiply by $(\Phi_i^{n+1}-\Phi_i^{n})\Delta x$ and sum over $i\in\II$:
\begin{equation*}
    \langle (D_x^cD_x^c \Phi^{n+1}- D_x^cD_x^c \Phi^n),\Phi^{n+1}-\Phi^n\rangle_2 = -\langle u_h^{n+1}-u_h^n, \Phi^{n+1}-\Phi^n\rangle_2.
\end{equation*}
After an integration by parts on the left hand side and plugging the continuity scheme \eqref{scheme_uh_lin} in the right hand side, one has
\begin{align*}
    \|D_x^c\Phi^{n+1}-D_x^c\Phi^n\|_2^2 &= -\Delta t \langle D_x^cJ_h^{n+1}, \Phi^{n+1}-\Phi^n\rangle_2 \\
    &\quad+ \frac{\lambda\Delta t\Delta x}{2}\langle (D_x^+D_x^-+D_x^-D_x^+)u_h^{n+1},\Phi^{n+1}-\Phi^n\rangle_2.
\end{align*}
Applying integrations by parts \eqref{IPP_centre} and \eqref{IPP2_decentre} on the right hand side together with Cauchy-Schwarz inequality, one obtains
\begin{align*}
    \|D_x^c\Phi^{n+1}-D_x^c\Phi^n\|_2^2 &= \Delta t \langle J_h^{n+1}, D_x^c\Phi^{n+1}-D_x^c\Phi^n\rangle_2 \\
    &\quad- 2\lambda\Delta t\Delta x\langle D_x^cu_h^{n+1},D_x^c\Phi^{n+1}-D_x^c\Phi^n\rangle_2\\
    &\leq\Delta t(\|J_h^{n+1}\|_2+2\lambda\Delta x\|D_x^cu_h^{n+1}\|_2)\|D_x^c\Phi^{n+1}-D_x^c\Phi^n\|_2.
\end{align*}
One can conclude by dividing both sides by $\|D_x^c\Phi^{n+1}-D_x^c\Phi^n\|_2$ and using estimates \eqref{estim_dxu_disc} and \eqref{estim_uh_dis}.
\end{proof}

In the following lemma, we establish that for $\delta >0$ small enough, the modified entropy functional is an equivalent $\|\cdot\|_\Delta$ norm.

\begin{Lemma}\label{lem_eqnorm_dis}
Under the assumptions of Theorem \ref{theo_hypoco_dis_lin} and assuming that $\Delta t\leq \Delta t_{\max}$, there exists $\alpha_1^*>0$ such that for all $n\geq 1$
\begin{equation*}
\begin{split}
    \left(\frac{1}{2}-\delta\,C_{J1}^*\,C_u^*\,C_P \right) \|F\|_\Delta^2&\leq H_\delta^\Delta[F^n]\\
    &\leq\left(\frac{1}{2}+\delta(\,C_{J1}^*\,C_u^*\,C_P+\,\alpha_1^*\,\Delta t_{\max}) \right)\|F^n\|_\Delta^2,
\end{split}
\end{equation*}
where $\alpha_1^*$ depends on $D_1^\Delta$, $D_2^\Delta$ and $\rhoinfs$.
\end{Lemma}

\begin{proof}
Let us start by estimating the second term of the  discrete modified entropy \eqref{def_entropy_dis}. A Cauchy-Schwarz inequality followed by \eqref{estim_Jh1_dis} and \eqref{estim_dxphi_dis} yield
\begin{align*}
    |\langle J_h^{n},D_x^c\Phi^{n}\rangle_2|&\leq\|J_h^n\|_2\|D_x^c \Phi^n\|_2\\
    &\leq C_{J1}^* \, C_P \, C_u^* \|F^n\|_\Delta\|\Pi^\Delta F^n\|_\Delta\\
    &\leq C_{J1}^* \, C_P \, C_u^* \|F^n\|_\Delta^2.
\end{align*}
The last term of \eqref{def_entropy_dis} shall be estimated using \eqref{estim_dtxphi_dis} and \eqref{estim_Jh1_dis}, together with Young inequality:
\begin{align*}
    \sum_{i\in\II}\frac{\Delta x}{2\Delta t} & \left((D_x^c\Phi^n)_i - (D_x^c\Phi^{n-1})_i\right)^2 \leq\frac{\Delta t}{2}\left(\|J_h^{n}\|_2+2\,\lambda\,C_u^*\,\|\Pi^\Delta F^{n}\|_\Delta\right)^2\\
    &\leq \Delta t\left((C_{J1}^*)^2+4(\lambda C_u^*)^2\right)\|F^{n}\|_\Delta^2.
\end{align*}
Setting 
\begin{equation*}
    \alpha_1^* = (C_{J1}^*)^2+ 4(\lambda C_u^*)^2 ,
\end{equation*}
one obtains
\begin{equation*}
    0\leq\sum_{i\in\II}\frac{\Delta x}{2\Delta t}\left((D_x^c\Phi^n)_i - (D_x^c\Phi^{n-1})_i\right)^2 \leq \alpha_1^*\Delta t\|F^n\|_\Delta^2.
\end{equation*}
Consequently, the modified entropy is bounded from below by a positive quantity as long as $\delta< \left(2C_{J1}^* \, C_u^* \, C_P\right)\inv\eqqcolon\delta_3$. Finally, the upper bound is obtained using that $\Delta t \leq \Delta t_{max}$.
\end{proof}

\begin{Proposition}\label{prop_evol_entropy_dis}
Under the assumptions of Theorem \ref{theo_hypoco_dis_lin}, there is $\delta_2>0$ such that for all $\Delta t\leq \Delta t_{\max}$ and $\delta\leq \delta_2$, there exists $K_{\delta}>0$ such that
\begin{equation}
    H_\delta^\Delta[F^{n+1}]-H_\delta^\Delta[F^n]\leq -\Delta t\,K_\delta\,\|F^{n+1}\|_\Delta^2.
\end{equation}
\end{Proposition}

\begin{proof}
Taking the difference between the modified entropy at time $t^{n+1}$ and $t^n$ we get
\begin{equation}\label{eq_dtH_disc}
    H_\delta[F^{n+1}] - H_\delta[F^{n}] = \frac{1}{2}\left(\|F^{n+1}\|_\Delta^2 -\|F^{n}\|_\Delta^2\right) + \delta T_1^n + \delta T_2^n,
\end{equation}
where
\begin{align*}
    &T_1^n = \sum_{i\in\II} \left(J_{h,i}^{n+1}(D_x^c\Phi^{n+1})_i- J_{h,i}^{n}(D_x^c\Phi^{n})_i\right) \Delta x,\\
    &T_2^n = \sum_{i\in\II}\frac{\Delta x}{2\Delta t}\left[\left((D_x^c\Phi^{n+1})_i - (D_x^c\Phi^{n})_i\right)^2 - \left((D_x^c\Phi^n)_i - (D_x^c\Phi^{n-1})_i\right)^2\right].
\end{align*}
We already showed through Lemma \ref{lem_coercivite_micro_dis} that
\begin{equation*}
    \frac{1}{2}\left(\|F^{n+1}\|_\Delta^2 -\|F^{n}\|_\Delta^2\right) \leq -\Delta t\,C_{mc}^*\|(I-\Pi^\Delta)F^n\|_\Delta^2.
\end{equation*}
It remains to deals with the last two terms. First of all, let us remark that $T_1^n=T_{11}^n+T_{12}^n+T_{13}^n$, with
\begin{align*}
    &T_{11}^n = \langle J_h^{n+1}-J_h^{n}, D_x^c\Phi^{n+1} \rangle_2, \\
    &T_{12}^n = \langle J_h^{n+1}, D_x^c\Phi^{n+1}-D_x^c\Phi^{n}\rangle_2, \\
    &T_{13}^n = \sum_{i\in\II}\Delta x \left[J_{h,i}^{n}(D_x^c\Phi^{n+1})_i - J_{h,i}^{n+1}(D_x^c\Phi^{n+1})_i + J_{h,i}^{n+1}(D_x^c\Phi^{n})_i -J_{h,i}^{n}(D_x^c\Phi^{n})_i \right].
\end{align*}
It is worth noticing that the first two terms are discrete equivalent of 
\begin{equation*}
    \left(\dt\langle J_h, \dx\Phi \rangle_{L^2(\TT)} + \langle J_h, \dtx\Phi\rangle_{L^2(\TT)}\right)\Delta t
\end{equation*}
and will therefore be treated in the same manner as their continuous counterparts. Let us first deal with the residual term $T_{13}^n$. After reorganizing the terms and integrating by parts, it becomes
\begin{align*}
    T_{13}^n &= \sum_{i\in\II}\Delta x (J_{h,i}^{n}-J_{h,i}^{n+1})\left((D_x^c\Phi^{n+1})_i-(D_x^c\Phi^{n})_i\right)\\
    &= \sum_{i\in\II}\Delta x \left((D_x^cJ_{h}^{n+1})_i-(D_x^cJ_{h}^{n})_i\right)(\Phi^{n+1}_i-\Phi^{n}_i).
\end{align*}
The next step is to replace the discrete space derivatives of $J_h$ using the continuity scheme \eqref{scheme_uh_lin}, followed by several uses of the auxiliary scheme \eqref{eq_aux_dis} and integrations by parts:
\begin{align*}
    T_{13}^n &= -\frac{1}{\Delta t}\sum_{i\in\II}\Delta x (u_{h,i}^{n+1}-2u_{h,i}^{n}+u_{h,i}^{n-1})(\Phi^{n+1}_i-\Phi^{n}_i)\\
    &\quad +\frac{\lambda\Delta x}{2}\sum_{i\in\II}\Delta x \left[\left((D_x^+D_x^- + D_x^-D_x^+)u_{h}^{n+1}\right)_i-\left((D_x^+D_x^- + D_x^-D_x^+)u_{h}^{n}\right)_i\right](\Phi^{n+1}_i-\Phi^{n}_i)\\
    &= \frac{1}{\Delta t}\sum_{i\in\II}\Delta x \left((D_x^cD_x^c\Phi^{n+1})_i-2(D_x^cD_x^c\Phi^{n})_i+(D_x^cD_x^c\Phi^{n-1})_i\right)(\Phi^{n+1}_i-\Phi^{n}_i)\\
    &\quad -2\lambda\Delta x\sum_{i\in\II}\Delta x \left((D_x^c u_{h}^{n+1})_i - (D_x^c u_{h}^{n})_i\right)\left((D_x^c\Phi^{n+1})_i-(D_x^c\Phi^{n})_i\right)\\
    &= -\frac{1}{\Delta t}\sum_{i\in\II}\Delta x \left( (D_x^c\Phi^{n+1})_i -2(D_x^c\Phi^{n})_i+ (D_x^c\Phi^{n-1})_i\right)\left((D_x^c\Phi^{n+1})_i-(D_x^c\Phi^{n})_i\right)\\
    &\quad -2\lambda\Delta x\sum_{i\in\II}\Delta x (u_{h,i}^{n+1}-u_{h,i}^{n})^2.\\
\end{align*}
The second term is clearly nonpositive. The first term is combined with $T_2^n$ using the identity $-a(a-b) + (a^2-b^2)/2 = -(a-b)^2/2$ with $a=(D_x^c\Phi^{n+1})_i-(D_x^c\Phi^{n})_i$ and $b=(D_x^c\Phi^{n})_i-(D_x^c\Phi^{n-1})_i$, therefore yielding $T_2^n + T_{13}^n \leq 0$, and finally
\begin{equation}\label{eq_dtH_disc2}
    H_\delta[F^{n+1}] - H_\delta[F^{n}] \leq \frac{1}{2}\left(\|F^{n+1}\|_\Delta^2 -\|F^{n}\|_\Delta^2\right) + \delta T_{11}^n + \delta T_{12}^n.
\end{equation}

The rest of the proof follows the same steps as in the continuous setting. The term $T_{11}^n$ can be expanded using scheme \eqref{scheme_Jh_lin} so that 
\begin{equation*}
    T_{11}^n = T_{111}^n + T_{112}^n + T_{113}^n + T_{114}^n,
\end{equation*}
where
\begin{align*}
    T_{111}^n &= -\Delta t\langle D_x^cS_h^{n+1}, D_x^c\Phi^{n+1} \rangle_2,\\
    T_{112}^n &= -\Delta tD_0\langle D_x^cu_h^{n+1}, D_x^c\Phi^{n+1} \rangle_2,\\
    T_{113}^n &= -\Delta t\langle (\rhoinfs)\inv J_f^{n+1}-\rhoinfs J_g^{n+1}, D_x^c\Phi^{n+1} \rangle_2,\\
    T_{114}^n &= \frac{\lambda\Delta t\Delta x}{2}\langle (D_x^+D_x^- + D_x^-D_x^+)J_h^{n+1},D_x^c\Phi^{n+1}\rangle_2.
\end{align*}
Thanks to Lemmas \ref{lem_mom_estim_disc} and \ref{lem_estim_phi_dis} we can estimate $T_{11k}$, $k=1,\dots,4$. After an integration by parts, using the auxiliary scheme \eqref{eq_aux_dis} and applying a Cauchy-Schwarz inequality, one has
\begin{align}
    \left|T_{111}^n\right| = \Delta t\left|\langle S_h^{n+1}, u_h^{n+1} \rangle_2\right|&\leq\Delta t \|S_h^{n+1}\|_2 \|u_h^{n+1}\|_2 \nonumber\\
    &\leq \Delta t C_{S}^* C_u^*\|(I-\Pi^\Delta)F^{n+1}\|_\Delta\|\Pi^\Delta F^{n+1}\|_\Delta.\label{estim_T111}
\end{align}
The second term is treated similarly and actually provides macroscopic coercivity:
\begin{align}
    T_{112}^n = -\Delta tD_0^\Delta\langle u_h^{n+1}, u_h^{n+1} \rangle_2&=-\Delta t D_0^\Delta\|u_h^{n+1}\|_2^2\leq -\Delta t \underline{D_0}(C_u^*)^2\|\Pi^\Delta F^{n+1}\|_\Delta^2.\label{estim_T112}
\end{align}
The next estimate follows directly from a Cauchy-Schwarz inequality:
\begin{align}
    \left|T_{113}^n\right| &= \Delta t\left|\langle  (\rhoinfs)\inv J_f^{n+1}-\rhoinfs J_g^{n+1}, D_x^c\Phi^{n+1}\rangle_2\right|\nonumber\\
    &\leq \Delta t \|(\rhoinfs)\inv J_f^{n+1}-\rhoinfs J_g^{n+1}\|_2 \|D_x^c\Phi^{n+1}\|_2\nonumber\\
    &\leq \Delta t C_{J2}^* C_P C_u^*\|(I-\Pi^\Delta)F^{n+1}\|_\Delta\|\Pi^\Delta F^{n+1}\|_\Delta.\label{estim_T113}
\end{align}
The estimation of $T_{114}^n$ is obtained through several integrations by parts, a Cauchy-Schwarz inequality and \eqref{estim_dxu_disc}:
\begin{align}
    \left|T_{114}^n\right| &= 2\lambda\Delta t\Delta x\left|\langle D_x^cJ_h^{n+1},D_x^cD_x^c\Phi^{n+1}\rangle_2\right|\nonumber\\
    &= 2\lambda\Delta t\Delta x\left|\langle D_x^cJ_h^{n+1},u_h^{n+1}\rangle_2\right|\nonumber\\
    &\leq 2\lambda\Delta t\Delta x \|J_h^{n+1}\|_2 \|D_x^cu_h^{n+1}\|_2\nonumber\\
    &\leq 2\lambda\Delta t C_{J1}^* C_u^*\|(I-\Pi^\Delta)F^{n+1}\|_\Delta\|\Pi^\Delta F^{n+1}\|_\Delta.\label{estim_T114}
\end{align}
It remains now to estimate $T_{12}^n$. After applying a Cauchy-Schwarz inequality followed by \eqref{estim_dtxphi_dis}, one obtains:
\begin{align}
    \left|T_{12}^n\right| &\leq \|J_h^{n+1}\|_2 \|D_x^c\Phi^{n+1}-D_x^c\Phi^{n}\|_2 \nonumber\\
    &\leq \Delta t(C_{J1}^*)^2\|(I-\Pi^\Delta)F^{n+1}\|_\Delta^2 + 2\Delta t\lambda C_{J1}^* C_u^*\|(I-\Pi^\Delta)F^{n+1}\|_\Delta\|\Pi^\Delta F^{n+1}\|_\Delta.\label{estim_T12}
\end{align}

Summarizing, combining estimates \eqref{estim_T111}, \eqref{estim_T112}, \eqref{estim_T113}, \eqref{estim_T114}, \eqref{estim_T12} and \eqref{estim_coercivite_micro_dis} in \eqref{eq_dtH_disc2}, it yields
\begin{equation}\label{estim_dtH_disc2}
\begin{aligned}
    H_\delta[F^{n+1}] - H_\delta[F^{n}] \leq &-\Delta t(C_{mc}^*-\delta(C_{J1}^*)^2)\|(I-\Pi^\Delta)F^{n+1}\|_\Delta^2 \\
    &-\Delta t\delta \underline{D_0}(C_u^*)^2\|\Pi^\Delta F^{n+1}\|_\Delta^2\\
    &\hspace{-6em}+\Delta t\delta(C_{S}^* C_u^* + C_{J2}^* C_P C_u^*+4\lambda C_{J1}^* C_u^*)\|(I-\Pi^\Delta)F^{n+1}\|_\Delta\|\Pi^\Delta F^{n+1}\|_\Delta
\end{aligned}
\end{equation}
Let us first set $\delta\in(0,\delta_1)$ with $\delta_1=C_{mc}^*(C_{J1}^*)^{-2}$ to ensure the positivity of $C_{mc}^*-\delta(C_{J1}^*)^2$. We can then rewrite the right hand side of \eqref{estim_dtH_disc2}.
\begin{equation}\label{estim_dtH_dis3}
\begin{aligned}
    H_\delta[F^{n+1}] - H_\delta[F^{n}] \leq &-\frac{\Delta t}{2} (C_{mc}^*-\delta(C_{J1}^*)^2)\|(I-\Pi^\Delta)F^{n+1}\|_\Delta^2\\
    &-\frac{\Delta t}{2}\delta \underline{D_0}(C_u^*)^2\|\Pi^\Delta F^{n+1}\|_\Delta^2\\
    &-\Delta t\|\Pi^\Delta F^{n+1}\|_\Delta^2P(\|(I-\Pi^\Delta)F^{n+1}\|_\Delta\|\Pi^\Delta F^{n+1}\|_\Delta^{-2}),
\end{aligned}
\end{equation}
where $P$ is the polynomial given by 
\begin{equation*}
    P(X) =\frac{1}{2}(C_{mc}^*-\delta\,(C_{J1}^*)^2)X^2 - \delta \widetilde{C}^\Delta X +\frac{1}{2}\delta\underline{D_0}(C_u^*)^2,
\end{equation*}
with $\widetilde{C}^\Delta=C_{S}^* C_u^* + C_{J2}^* C_P C_u^*+4\lambda C_{J1}^* C_u^*$. 

The sum of the first two terms in \eqref{estim_dtH_dis3} is nonpositive thanks to our choice of $\delta$. It remains to impose a second condition on $\delta$ to ensure that the polynomial $P$ is positive. Since the leading order coefficient of $P$ is positive, it suffices to choose $\delta$ in such a way that the discriminant is negative. It yields the following condition on $\delta$: 
\begin{equation}
    \delta_2\coloneqq\frac{C_{mc}^* \underline{D_0} (C_u^*)^2}{(\widetilde{C}^\Delta)^2+ (C_{J1}^*)^2 \underline{D_0} (C_u^*)^2}>\delta.
\end{equation}
Assuming that $\delta\in(0,\min(\delta_1,\delta_2))$, one then has
\begin{equation*}
    H_\delta[F^{n+1}] - H_\delta[F^{n}] \leq -\frac{\Delta t}{2}\left[ (C_{mc}^*-\delta(C_{J1}^*)^2)\|(I-\Pi^\Delta)F^{n+1}\|_\Delta^2 +\delta \underline{D_0}(C_u^*)^2\|\Pi^\Delta F^{n+1}\|_\Delta^2\right].
\end{equation*}
Finally, setting $K_\delta=\frac{1}{2}\min(C_{mc}^*-\delta(C_{J1}^*)^2,\delta \underline{D_0}(C_u^*)^2)$ and using orthogonality properties, one concludes the proof.
\end{proof}
\begin{proof}[Proof of Theorem \ref{theo_hypoco_dis_lin}]
Starting from the result of Proposition \ref{prop_evol_entropy_dis}, we use Lemma \ref{lem_eqnorm_dis} to bound $\|F^{n+1}\|_\Delta^2$ from below by $H_\delta[F^{n+1}]$ and obtain
\begin{equation*}
    (1+\Delta t\,\kappa)H_\delta^\Delta[F^{n+1}]\leq H_\delta^\Delta[F^n],
\end{equation*}
where we set $\kappa=\frac{K_\delta}{C_\delta}$ and $C_\delta=\frac{1}{2}+\delta\, C_{J1}^*\,C_u^*\,C_P+\delta\,\alpha_1^*\,\Delta t_{\max}$. It implies that for all $n\geq 1$, 
\begin{equation*}
    H_\delta^\Delta[F^{n+1}]\leq H_\delta^\Delta[F^1](1+\Delta t \kappa)^{-n}=H_\delta^\Delta[F^1]\exp\left(-t^{n}\log(1+\Delta t \kappa)/\Delta t\right).
\end{equation*}
Since $s\mapsto \log(1+s\kappa)/s$ is nonincreasing on $]0,+\infty[$ and $\Delta t\leq \Delta t_{max}$, we obtain
\[H_\delta^\Delta[F^{n+1}]\leq H_\delta^\Delta[F^1]\,e^{-\beta\,t^{n}},\]
where $\beta=\log(1+\Delta t_{max}\kappa)/\Delta t_{max}$. Then, using lemmas \ref{lem_eqnorm_dis} and \ref{lem_coercivite_micro_dis}, we finally get  
\[H_\delta^\Delta[F^{n+1}]\leq C_\delta \|F^1\|^2_\Delta \,e^{-\beta\,t^n}\leq C_\delta \|F^0\|^2_\Delta  \,e^{-\beta\,t^{n}}.\]
Finally, choosing 
\begin{equation*}
    \delta\in(0,\min(\delta_1,\delta_2,\delta_3)),
\end{equation*}
the constant on the left hand side in Lemma \ref{lem_eqnorm_dis} is positive and one can conclude the proof.
\end{proof}


\section{The nonlinear problem}\label{sec:discrete_nonlin}

In this section, we prove that approximate solutions to the nonlinear system with initial data sufficiently close to equilibrium converge exponentially towards this equilibrium as time goes to infinity. The idea of the proof is as follows: we first establish in Section \ref{sec:existence_bornes} that solutions to the numerical scheme \eqref{scheme_f_nonlin}--\eqref{scheme_g_nonlin} satisfy the maximum principle (and also that they exist), and then in Section \ref{sec:hypoco_nonlin}, using the fact that the entropy dissipation for the nonlinear problem is a small perturbation of the entropy dissipation of the linearized problem for initial data sufficiently close to equilibrium, the nonlinear discrete hypocoercivity is obtained.

\subsection{Existence and maximum principle}\label{sec:existence_bornes}

This section is devoted to the proof of the following result.

\begin{Theorem}\label{theo_existence_bornes}
Under assumptions \eqref{hyp_chi_dis}, let $\rhoinfs$ be defined by \eqref{def_rhoinf_dis_full} and assume that there exists positive constants $\gamma_1 < \rhoinfs$ and $\gamma_2$ such that the initial datum $\Fg^0=(\fg_{ij}^0,\ggot_{ij}^0)_{i\in\II,j\in\J}$ satisfies for all $i \in \II$, $j \in\J$
\begin{align*}
   & (\rhoinfs-\gamma_1)\,\chi_{1,j}\leq \fg_{ij}^0\leq (\rhoinfs+\gamma_2)\,\chi_{1,j},\\
   & (\rhoinfs+\gamma_2)\inv\,\chi_{2,j}\leq \ggot_{ij}^0\leq (\rhoinfs-\gamma_1)\inv\,\chi_{2,j}.
\end{align*}
Assuming moreover that $\lambda=\Delta x/2\,\Delta t \geq v^*/2$, the scheme \eqref{scheme_f_nonlin}--\eqref{scheme_g_nonlin} admits a solution $(\fg_{ij}^n,\ggot_{ij}^n)_{i\in\II,j\in\J,n\geq0}$ such that for all $i\in\II$, $j\in\J$, $n\geq 0$,
\begin{align}
   & (\rhoinfs-\gamma_1)\,\chi_{1,j}\leq \fg_{ij}^n\leq (\rhoinfs+\gamma_2)\,\chi_{1,j},\label{bornes_f_dis}\\
   & (\rhoinfs+\gamma_2)\inv\,\chi_{2,j}\leq \ggot_{ij}^n\leq (\rhoinfs-\gamma_1)\inv\,\chi_{2,j}.\label{bornes_g_dis}
\end{align}
\end{Theorem}

 \begin{Remark}
 The condition $\lambda\geq v^*/2$ is there to ensure the monotonicity of the Lax-Friedrichs fluxes, which is necessary in our proof. In practice, we observe that this restriction is  necessary in the numerical results presented in Section \ref{sec:numeric}.
 \end{Remark}

To prove this result, we introduce the following truncated version of the scheme \eqref{scheme_f_nonlin}--\eqref{scheme_g_nonlin}:
\begin{align}
    & \frac{\fg_{ij}^{n+1}-\fg_{ij}^n}{\Delta t}+\frac{1}{\Delta x\,\Delta v}\left(\mathcal{F}_{\iph,}^{n+1}-\mathcal{F}_{\imh,j}^{n+1}\right)=\chi_{1,j}-\tilde{\rho}_{\ggot,i}^{n+1}\tilde{\fg}_{ij}^{n+1} ,\label{scheme_f_nonlin_trunc}\\
    & \frac{\ggot_{ij}^{n+1}-\ggot_{ij}^n}{\Delta t}+\frac{1}{\Delta x\,\Delta v}\left(\mathcal{G}_{\iph,j}^{n+1}-\mathcal{G}_{\imh,j}^{n+1}\right)=\chi_{2,j}-\tilde{\rho}_{\fg,i}^{n+1}\tilde{\ggot}_{ij}^{n+1},\label{scheme_g_nonlin_trunc}
\end{align}
with the Lax-Friedrichs fluxes \eqref{flux_F}--\eqref{flux_G}, and where the truncated quantities are defined by
\begin{equation}\label{def_f_trunc}
    \tilde{\fg}_{ij}\coloneqq\left\{\begin{array}{ll}
        (\rhoinfs-\gamma_1)\chi_{1,j} & \text{ if } \fg_{ij}\leq (\rhoinfs-\gamma_1)\chi_{1,j}  \\
        \fg_{ij} & \text{ if } (\rhoinfs-\gamma_1)\chi_{1,j}\leq \fg_{ij}\leq (\rhoinfs+\gamma_2)\chi_{1,j}\\
        (\rhoinfs+\gamma_2)\chi_{1,j} & \text{ if } \fg_{ij} \geq (\rhoinfs+\gamma_2)\chi_{1,j},
    \end{array}\right.
\end{equation}
and
\begin{equation}\label{def_g_trunc}
    \tilde{\ggot}_{ij}\coloneqq\left\{\begin{array}{ll}
        (\rhoinfs+\gamma_2)\inv\chi_{2,j} & \text{ if } \ggot_{ij}\leq (\rhoinfs+\gamma_2)\inv\chi_{2,j}  \\
        \ggot_{ij} & \text{ if } (\rhoinfs+\gamma_2)\inv\chi_{2,j}\leq \ggot_{ij}\leq (\rhoinfs-\gamma_1)\inv\chi_{2,j}\\
        (\rhoinfs-\gamma_1)\inv\chi_{2,j} & \text{ if } \ggot_{ij} \geq (\rhoinfs-\gamma_1)\inv\chi_{2,j}.
    \end{array}\right.
\end{equation}
The corresponding truncated densities are then given by
\begin{equation*}
    \tilde{\rho}_{\fg,i}=\sum_{j\in\J}\Delta v\,\tilde{\fg}_{ij},\qquad \tilde{\rho}_{\ggot,i}=\sum_{j\in\J}\Delta v\,\tilde{\ggot}_{ij} \qquad \forall i\in\II.
\end{equation*}

To prove Theorem \ref{theo_existence_bornes}, we start by establishing the existence of a solution to the truncated scheme (Lemma \ref{lem_existence_trunc}) and then the fact that a solution to the truncated scheme verifies the estimates \eqref{bornes_f_dis}--\eqref{bornes_g_dis} (Lemma \ref{lem_bornes_trunc}).

\begin{Lemma}\label{lem_existence_trunc}
Given any $(\fg_{ij}^n,\ggot_{ij}^n)_{i\in\II,j\in\J}$, the truncated scheme \eqref{scheme_f_nonlin_trunc}--\eqref{scheme_g_nonlin_trunc} admits at least one solution $(\fg_{ij}^{n+1},\ggot_{ij}^{n+1})_{i\in\II,j\in\J}$.
\end{Lemma}

\begin{proof}
The proof relies on the following result which is proven for example in \cite[Lemma 1.4, Chapter II]{temam1984}.
\begin{Lemma}\label{lem_existence_Navier}
    Let $X$ be a finite dimensional Hilbert space with scalar product $\langle\cdot,\cdot\rangle_X$ and associated norm $\|\cdot\|_X$. Let $P:X\rightarrow X$ be a continuous mapping such that $\langle P(\xi),\xi\rangle_X >0$ for all $\xi\in X$ such that $\|\xi\|_X=k$ for some fixed $k>0$. Then there exists $\xi_0\in X$ such that $\|\xi_0\|_X\leq k$ and $P(\xi)=0$.
\end{Lemma}
We apply this result with $X=\left(\R^{(2L+1)N} \right)^2$ where for $u^1=(f^1,g^1) \in X$ and $u^2=(f^2,g^2)\in X$,
\begin{equation*}
    \langle u^1_\Delta,u^2_\Delta\rangle_X=\langle f^1_\Delta,f^2_\Delta\rangle+\langle g^1_\Delta,g^2_\Delta\rangle,
\end{equation*}
$\langle \cdot,\cdot \rangle$ being the classical Euclidian dot product on $\R^{(2L+1)N}$.
Then, let us define
\begin{equation*}
    P:\begin{pmatrix} \fg\\\ggot \end{pmatrix}\mapsto
    \begin{pmatrix} \Delta x\Delta v(\fg_{ij}-\fg_{ij}^n)+\Delta t(\mathcal{F}_{\iph,j}-\mathcal{F}_{\imh,j})-\Delta x\Delta v\Delta t(\chi_{1,j}-\tilde{\rho}_{\ggot,i}\tilde{\fg}_{ij})\\ 
    \Delta x\Delta v(\ggot_{ij}-\ggot_{ij}^n)+\Delta t(\mathcal{G}_{\iph,j}-\mathcal{G}_{\imh,j})-\Delta x\Delta v\Delta t(\chi_{1,j}-\tilde{\rho}_{\fg,i}\tilde{\ggot}_{ij})
    \end{pmatrix},
\end{equation*}
where $\mathcal{F}$ and $\mathcal{G}$ are the numerical fluxes \eqref{flux_F}--\eqref{flux_G} computed with $\fg_{ij}$ and $\ggot_{ij}$. Let us now show that
\begin{equation*}
    \left\langle P(\begin{pmatrix} \fg\\\ggot \end{pmatrix}), \begin{pmatrix} \fg\\\ggot \end{pmatrix} \right\rangle_X > 0, \quad \text{ for all } \begin{pmatrix} \fg\\\ggot \end{pmatrix} \text{ such that }\left\|\begin{pmatrix} \fg\\\ggot \end{pmatrix}\right\|_X=k,
\end{equation*}
where $k$ is taken large enough. This scalar product splits into 3 terms, 
\begin{equation*}
    \left\langle P(\begin{pmatrix} \fg\\\ggot \end{pmatrix}), \begin{pmatrix} \fg\\\ggot \end{pmatrix} \right\rangle_X=A_1+A_2+A_3,
\end{equation*} 
where
\begin{align*}
    A_1 &= \sum_{(i,j)\in\II\times\J} \Delta x\Delta v\left((\fg_{ij}-\fg_{ij}^n)\fg_{ij} + (\ggot_{ij}-\ggot_{ij}^n)\ggot_{ij}\right),\\
    A_2 &= \Delta t\sum_{(i,j)\in\II\times\J} \left((\mathcal{F}_{\iph,j}-\mathcal{F}_{\imh,j})\fg_{ij} + (\mathcal{G}_{\iph,j}-\mathcal{G}_{\imh,j})\ggot_{ij}\right),\\
    A_3 &= -\Delta t\sum_{(i,j)\in\II\times\J} \Delta x\Delta v\left((\chi_{1,j}-\tilde{\rho}_{\ggot,i}\tilde{\fg}_{ij})\fg_{ij} + (\chi_{2,j}-\tilde{\rho}_{\fg,i}\tilde{\ggot}_{ij})\ggot_{ij}\right).
\end{align*}
Using the relation $a(a-b)\geq(a^2-b^2)/2$ one gets
\begin{equation}\label{estim_A1}
    A_1 \geq\frac{1}{2}\left( \left\| \begin{pmatrix} \fg\\\ggot \end{pmatrix}\right\|_X^2 - \left\| \begin{pmatrix} \fg^n\\\ggot^n \end{pmatrix}\right\|_X^2\right).
\end{equation}
Then, by definition of the numerical fluxes, one can use the same computations as in the proof of Lemma \ref{lem_coercivite_micro_dis} to obtain
\begin{equation}\label{estim_A2}
    A_2 = 2\lambda\Delta t\Delta x\sum_{(i,j)\in\II\times\J} \left((D_x^c \fg_j)_i^2 + (D_x^c \ggot_j)_i^2\right) \Delta x\Delta v \geq 0.
\end{equation}
Now, by assumptions \eqref{hyp_chi_dis}, the velocity profiles are bounded:
\begin{equation*}
    \exists C_\chi>0 \quad\text{ such that }\quad \forall j\in\J, \quad 0\leq \chi_{1,j},\chi_{2,j}\leq C_\chi.
\end{equation*}
Then, by definition of the truncated quantities, there exists a constant $C>0$ such that
\begin{equation*}
    |\chi_{1,j}-\tilde{\rho}_{\ggot,i}\tilde{\fg}_{ij}|\leq C\quad\text{and}\quad |\chi_{2,j}-\tilde{\rho}_{\fg,i}\tilde{\ggot}_{ij}|\leq C.
\end{equation*}
Applying a Cauchy-Schwarz inequality on $A_3$, one gets
\begin{equation}\label{estim_A3}
    A_3\geq -\Delta t C \left\| \begin{pmatrix} \fg\\\ggot \end{pmatrix}\right\|_X.
\end{equation}
Combining \eqref{estim_A1}, \eqref{estim_A2} and \eqref{estim_A3} yields the estimate
\begin{equation*}
    \left\langle P(\begin{pmatrix} \fg\\\ggot \end{pmatrix}), \begin{pmatrix} \fg\\\ggot \end{pmatrix} \right\rangle_X\geq 
    \frac{1}{2}\left( \left\| \begin{pmatrix} \fg\\\ggot \end{pmatrix}\right\|_X^2 - \left\| \begin{pmatrix} \fg^n\\\ggot^n \end{pmatrix}\right\|_X^2\right) -\Delta t C \left\| \begin{pmatrix} \fg\\\ggot \end{pmatrix}\right\|_X.
\end{equation*} 
The right hand side is a second order polynomial in $\left\| \begin{pmatrix} \fg\\\ggot \end{pmatrix}\right\|_X$ with a positive leading coefficient. Since $\left\| \begin{pmatrix} \fg^n \\\ggot^n \end{pmatrix}\right\|_X$ is a constant in this context, there exists $k>0$ such that if $\left\| \begin{pmatrix} \fg\\\ggot \end{pmatrix}\right\|_X\geq k$ then 
\begin{equation*}
    \left\langle P(\begin{pmatrix} \fg\\\ggot \end{pmatrix}), \begin{pmatrix} \fg\\\ggot \end{pmatrix} \right\rangle_X> 0.
\end{equation*} 
Finally, we can apply Lemma \ref{lem_existence_Navier} to obtain existence of $\begin{pmatrix} \fg^{n+1} \\\ggot^{n+1} \end{pmatrix}$ such that $P(\begin{pmatrix} \fg^{n+1} \\\ggot^{n+1} \end{pmatrix})=0$, therefore ensuring existence of a solution to the truncated scheme \eqref{scheme_f_nonlin_trunc}-\eqref{scheme_g_nonlin_trunc}.
\end{proof}

\begin{Lemma}\label{lem_bornes_trunc}
If  $(\fg_{ij}^n,\ggot_{ij}^n)_{i\in\II,j\in\J}$ satisfies the estimates \eqref{bornes_f_dis}--\eqref{bornes_g_dis}, then any solution\\ $(\fg_{ij}^{n+1},\ggot_{ij}^{n+1})_{i\in\II,j\in\J}$ to the truncated scheme \eqref{scheme_f_nonlin_trunc}--\eqref{scheme_g_nonlin_trunc} also satisfies these estimates.
\end{Lemma}

\begin{proof}
Let us focus on showing that a solution to the nonlinear truncated scheme \eqref{scheme_f_nonlin_trunc}-\eqref{scheme_g_nonlin_trunc} satisfies
\begin{equation*}
    \fg_{ij}^{n+1} \geq (\rhoinfs-\gamma_1)\chi_{1,j}.
\end{equation*}
We start by setting $(i,j)\in\II\times\J$ such that 
\begin{equation}\label{eq_ij_bounds}
    \fg_{ij}^{n+1} - (\rhoinfs-\gamma_1)\chi_{1,j} = \underset{(k,l)\in\II\times\J}{\min}(\fg_{kl}^{n+1} - (\rhoinfs-\gamma_1)\chi_{1,l}).
\end{equation}
Our aim is now to show that this quantity is nonnegative. We begin by multiplying the line of \eqref{scheme_f_nonlin_trunc} corresponding to the fixed pair $(i,j)$ by
\begin{equation}\label{eq_negPart}
    (\fg_{ij}^{n+1} - (\rhoinfs-\gamma_1)\chi_{1,j})^-=\min(0,\fg_{ij}^{n+1} - (\rhoinfs-\gamma_1)\chi_{1,j})\leq 0.
\end{equation}
It yields an expression of the form $B_1=B_2+B_3$ where
\begin{align*}
    B_1 &= \Delta x\Delta v(\fg_{ij}^{n+1}-\fg_{ij}^{n})(\fg_{ij}^{n+1} - (\rhoinfs-\gamma_1)\chi_{1,j})^-,\\
    B_2 &= -\Delta t(\mathcal{F}_{\iph,j}^{n+1}-\mathcal{F}_{\imh,j}^{n+1})(\fg_{ij}^{n+1} - (\rhoinfs-\gamma_1)\chi_{1,j})^-,\\
    B_3 &=  \Delta t\Delta x\Delta v(\chi_{1,j}-\tilde{\rho}_{\ggot,i}^{n+1}\tilde{\fg}_{ij}^{n+1})(\fg_{ij}^{n+1} - (\rhoinfs-\gamma_1)\chi_{1,j})^-.
\end{align*}
Starting with $B_1$, we add and subtract $(\rhoinfs-\gamma_1)\chi_{1,j}$ to obtain
\begin{equation*}
    B_1 = \Delta x\Delta v( (\fg_{ij}^{n+1}-(\rhoinfs-\gamma_1)\chi_{1,j}) - (\fg_{ij}^{n}-(\rhoinfs-\gamma_1)\chi_{1,j}))(\fg_{ij}^{n+1} - (\rhoinfs-\gamma_1)\chi_{1,j})^-
\end{equation*}
Then, under the hypothesis that $\fg^n_{ij}$ satisfies the bounds \eqref{bornes_f_dis} and by definition \eqref{eq_negPart} of the negative part, 
\begin{equation}\label{estim_B1}
    B_1 \geq \Delta x\Delta v\,(\fg_{ij}^{n+1}-(\rhoinfs-\gamma_1)\chi_{1,j})(\fg_{ij}^{n+1} - (\rhoinfs-\gamma_1)\chi_{1,j})^- \geq 0.
\end{equation}
Next, we turn our attention to $B_3$. We need to consider two cases:
\begin{itemize}
    \item If $\fg_{ij}^{n+1}\geq(\rhoinfs-\gamma_1)\chi_{1,j}$ then $B_3=0$.
    \item Else, by definition of the truncated quantities \eqref{def_f_trunc} and \eqref{def_g_trunc}, one has $\tilde{\fg}_{ij}^{n+1} = (\rhoinfs-\gamma_1)\chi_{1,j}$ and since the discrete velocity profiles have unit mass, $\tilde{\rho}_{\ggot,i}^{n+1}\leq(\rhoinfs-\gamma_1)\inv$. Therefore, 
    \begin{equation*}
        \chi_{1,j}-\tilde{\rho}_{\ggot,i}^{n+1}\tilde{\fg}_{ij}^{n+1} \geq \chi_{1,j}-\frac{\rhoinfs-\gamma_1}{\rhoinfs-\gamma_1}\chi_{1,j}\geq 0.
    \end{equation*}
\end{itemize}
As a result, one gets
\begin{equation}\label{estim_B3}
    B_3\leq 0.
\end{equation}
Finally, the sign of $B_2$ relies on the monotonicity of the numerical flux. More precisely, let us consider a monotone numerical flux in a general two-point approximation form: 
\begin{equation*}
    \mathcal{F}_{\iph,j}^{n+1}=\varphi_j(\fg_{ij}^{n+1},\fg_{i+1,j}^{n+1}),
\end{equation*}
where $a\mapsto \varphi_j(a,\cdot)$ is assumed to be nondecreasing and $b\mapsto \varphi_j(\cdot,b)$ is assumed to be nonincreasing. Then, the balance of fluxes at cell $\mathcal{X}_i$ rewrites as
\begin{align}\label{eq_fluxmono_varphi}
    \mathcal{F}_{\iph,j}^{n+1}-\mathcal{F}_{\imh,j}^{n+1}=\, &\varphi_j(\fg_{ij}^{n+1},\fg_{i+1,j}^{n+1}) - \varphi_j(\fg_{ij}^{n+1},\fg_{i,j}^{n+1})\nonumber\\
    & + \varphi_j(\fg_{ij}^{n+1},\fg_{ij}^{n+1}) - \varphi_j(\fg_{i-1,j}^{n+1},\fg_{ij}^{n+1}).
\end{align}
From our choice \eqref{eq_ij_bounds} of the pair $(i,j)$, we deduce that
\begin{equation*}
    \fg_{ij}^{n+1} - (\rhoinfs-\gamma_1)\chi_{1,j} \leq \fg_{kl}^{n+1} - (\rhoinfs-\gamma_1)\chi_{1,l},\quad\forall (k,l)\in\II\times\J,
\end{equation*}
therefore in particular, for every $k\in\II$, $\fg_{ij}^{n+1}\leq\fg_{kj}^{n+1}$. Consequently, due to the monotonicity of the function $\varphi_j$, it yields
\begin{equation}\label{estim_B2}
    B_2 \leq 0.
\end{equation}
At this point, one could select any flux satisfying the monotonicity property. In our framework, we opted for the Lax-Friedrichs flux
\begin{equation*}
    \varphi_j(a,b)=\frac{v_j\Delta v}{2}(a+b) - \lambda\Delta v(b-a)
\end{equation*}
which is monotone under condition. More precisely, we want to ensure the nonnegativity of the first partial derivative of $\varphi_j$ and the nonpositivity of the second. Therefore, one imposes $\lambda=\Delta x/2\,\Delta t$ such that
\begin{align*}
    &\partial_a \varphi_j(a,b) \geq 0 \iff \frac{v_j}{2} + \lambda \geq 0 \iff -\frac{v_j}{2} \leq \lambda,\\
    &\partial_b \varphi_j(a,b) \leq 0 \iff \frac{v_j}{2} -\lambda \geq 0 \iff \frac{v_j}{2} \leq \lambda.
\end{align*}
In the end, the choice $\lambda\geq v_\star/2$ suffices.

Gathering \eqref{estim_B1}, \eqref{estim_B3} and \eqref{estim_B2} into $B_1=B_2+B_3$, we get that $B_1=0$, and conseuquently that
\begin{equation*}
    \fg_{i,j}^{n+1} - (\rhoinfs-\gamma_1)\chi_{1,j} \geq 0.
\end{equation*}
Lastly, since we chose $(i,j)$ satisfying \eqref{eq_ij_bounds}, 
\begin{equation*}
    \fg_{i,j}^{n+1} \geq (\rhoinfs-\gamma_1)\chi_{1,j},\quad \forall (i,j)\in\II\times\J.
\end{equation*}
The remaining bounds can then be obtained following the same steps.
\end{proof}

\begin{proof}[Proof of Theorem \ref{theo_existence_bornes}]
We proceed by induction on $n$. The case $n=0$ is satisfied by assumption. Suppose now that there exists $(\fg^n,\ggot^n)$ satisfying bounds \eqref{bornes_f_dis} and \eqref{bornes_g_dis}. Then, we obtain from Lemma \ref{lem_existence_trunc} the existence of $(\fg^{n+1},\ggot^{n+1})$ solution to the truncated scheme \eqref{scheme_f_nonlin_trunc}-\eqref{scheme_g_nonlin_trunc}, which satisfies \eqref{bornes_f_dis} and \eqref{bornes_g_dis} according to Lemma \ref{lem_bornes_trunc}. But then, we have
\begin{align*}
    &\tilde{\fg}^{n+1} = \fg^{n+1},\quad \tilde{\ggot}^{n+1} = \ggot^{n+1},\\
    &\tilde{\rho_\fg}^{n+1} = \rho_\fg^{n+1},\quad \tilde{\rho_\ggot}^{n+1} = \rho_\ggot^{n+1},
\end{align*}
meaning that $(\fg^{n+1},\ggot^{n+1})$ is indeed a solution to the original nonlinear scheme \eqref{scheme_f_nonlin}-\eqref{scheme_g_nonlin}, which concludes the proof.
\end{proof}

\subsection{Local hypocoercivity result}\label{sec:hypoco_nonlin}

Thanks to the maximum principle estimates obtained in Theorem \ref{theo_existence_bornes}, we can now extend the decay result for the discrete linearized problem to a local result in the nonlinear case with the same method.

\begin{Theorem}\label{theo_hypoco_dis_nonlin}
Under the assumptions of Theorem \ref{theo_hypoco_dis_lin} and Theorem \ref{theo_existence_bornes}, with $\gamma_1$ and $\gamma_2$ small enough, a solution $\Fg^n=(\fg_{ij}^n,\ggot_{ij}^n)_{i\in\II,j\in\J}$ of the nonlinear scheme \eqref{scheme_f_nonlin}--\eqref{scheme_g_nonlin} satisfies for all $n\geq 0$
\begin{equation}
    \label{eq:expDecayNL}
    \|\Fg^n-\Fg^\infty\|_\Delta\leq C\,\|\Fg^0-\Fg^\infty\|_\Delta\,e^{-\kappa\,t^n},
\end{equation}
where the equilibrium $\Fg^\infty$ is defined by \eqref{def_eq_dis}. 

The constants $C$ and $\kappa$ do not depend on the size of the discretization $\Delta$.
\end{Theorem}

\begin{proof}
Let us first denote by $Q^\Delta(\fg,\ggot)$ the discrete nonlinear collision operator given by
\begin{equation}\label{eq_Q_disc}
    Q^\Delta(\fg,\ggot) = \begin{pmatrix}\chi_{1,j} - \rho_{\ggot,i}\fg_{ij} \\ \chi_{2,j} - \rho_{\fg,i}\ggot_{ij}\end{pmatrix}.
\end{equation}
Let us now rewrite the nonlinear scheme \eqref{scheme_f_nonlin}-\eqref{scheme_g_nonlin} in terms of $\widetilde{\Fg}_{ij}^n = \Fg_{ij}^n - \Fg_{j}^{\infty}$. Since the equilibrium $\Fg_{j}^{\infty}$ does not depend on position nor time, the discrete time derivative of $\widetilde{\Fg}$ satisfies
\begin{align}
    \frac{\widetilde{\Fg}_{ij}^{n+1}-\widetilde{\Fg}_{ij}^{n}}{\Delta t} +\frac{1}{\Delta x\,\Delta v} \begin{pmatrix}
    \widetilde{\mathcal{F}}_{\iph,j}^{n+1}-\widetilde{\mathcal{F}}_{\imh,j}^{n+1} \\
    \widetilde{\mathcal{G}}_{\iph,j}^{n+1}-\widetilde{\mathcal{G}}_{\imh,j}^{n+1}
    \end{pmatrix}
    = Q^\Delta(\fg^{n+1},\ggot^{n+1}),
\end{align}
where $\widetilde{\mathcal{F}}$ and $\widetilde{\mathcal{G}}$ are the numerical fluxes applied to $\widetilde{\fg}$ and $\widetilde{\ggot}$ respectively. In addition, one can compute the following relation
\begin{equation}\label{eq_QFminusLFt}
    Q^\Delta(\fg,\ggot) - \LL^\Delta\widetilde{\Fg} =\begin{pmatrix}
    -(\rho_{\ggot,i}-(\rhoinfs)\inv)(\fg - \rhoinfs\chi_{1,j})\\
    -(\rho_{\fg,i}-\rhoinfs)(\ggot -(\rhoinfs)\inv\chi_{2,j})
    \end{pmatrix}=\begin{pmatrix}
    -(\rho_{\ggot,i}-(\rhoinfs)\inv)\fgt\\
    -(\rho_{\fg,i}-\rhoinfs)\tilde{\ggot}
    \end{pmatrix}.
\end{equation}
In addition, multiplying by $\Delta v$ and summing over $j\in\J$ the bounds \eqref{bornes_f_dis} and \eqref{bornes_g_dis} obtained in Theorem \ref{theo_existence_bornes}, one gets
\begin{equation*}
    \left\lbrace\begin{aligned}
    \rhoinfs-\gamma_1 &\leq \rho_{\fg,i}^{n+1} \leq \rhoinfs+\gamma_2,\\
    (\rhoinfs+\gamma_2)\inv &\leq \rho_{\ggot,i}^{n+1} \leq (\rhoinfs-\gamma_1)\inv.
    \end{aligned}\right.
\end{equation*}
Therefore, by setting 
\begin{equation*}
    \gamma_3=\max(\gamma_1,\gamma_2) \quad\text{and}\quad \gamma_4=\max\left(\frac{\gamma_2}{\rhoinfs(\rhoinfs+\gamma_2)},\frac{\gamma_1}{\rhoinfs(\rhoinfs-\gamma_1)}\right),
\end{equation*}
one has for all $i\in\II$
\begin{equation}\label{eq_bounds_densities}
    |\rho_{\fg,i}^{n+1}-\rhoinfs| \leq\gamma_3 \quad\text{and}\quad |\rho_{\ggot,i}^{n+1}-(\rhoinfs)\inv| \leq\gamma_4.
\end{equation}
We are now able to estimate the norm of \eqref{eq_QFminusLFt}. Using the previous estimation \eqref{eq_bounds_densities} and setting $\gamma=\max(\gamma_3,\gamma_4)$ we obtain
\begin{equation}\label{estim_QFminusLFt_norm}
    \left\| Q(\fg^{n+1},\ggot^{n+1})-\LL^\Delta\widetilde{\Fg}^{n+1}\right\|_\Delta \leq \gamma\|\widetilde{\Fg}^{n+1}\|_\Delta.
\end{equation}

Then, by combining this estimate with Lemma \ref{lem_coercivite_micro_dis}, we obtain the following dissipation estimate
\begin{equation}\label{eq_estim_dissip_nonlin}
    \frac{1}{2}\left(\|\widetilde{\Fg}^{n+1}\|_\Delta^2-\|\widetilde{\Fg}^{n}\|_\Delta^2\right) \leq -\Delta tC_{mc}^*\|(I-\Pi^\Delta)\widetilde{\Fg}^{n+1}\|^2_\Delta + \Delta t\gamma\|\widetilde{\Fg}^{n+1}\|^2_\Delta.
\end{equation}
Defining, $\hgt\coloneqq\fgt-\ggt$, the rest of the proof unfolds as in the linear setting, except a slight modification on the right hand side of the moment scheme on $J_{\hgt}$:
\begin{align}
    & \frac{u_{\hgt,i}^{n+1}-u_{\hgt,i}^n}{\Delta t}+(D_x^cJ_{\hgt}^{n+1})_i-\frac{\Delta x\,\lambda}{2}\left((D_x^+D_x^-+D_x^-D_x^+)u_{\hgt}^{n+1}\right)_i=0, \label{scheme_uh_nonlin}\\
    &  \frac{J_{\hgt,i}^{n+1}-J_{\hgt,i}^n}{\Delta t}+(D_x^cS_{\hgt}^{n+1})_i+D_0^\Delta\,(D_x^cu_{\hgt}^{n+1})_i-\frac{\Delta x\,\lambda}{2}\left((D_x^+D_x^-+D_x^-D_x^+)J_{\hgt}^{n+1}\right)_i\nonumber \\
    & \qquad =-((\rhoinfs)\inv\,J_{\fgt,i}^{n+1}-\rhoinfs\,J_{\ggt,i}^{n+1})-\left((\rho_{\ggot,i}^{n+1}-(\rhoinfs)\inv)J_{\fgt,i}^{n+1}-(\rho_{\fg,i}^{n+1}-\rhoinfs)J_{\tilde{\ggot},i}^{n+1}\right).\label{scheme_Jh_nonlin}
\end{align}
Using the bounds \eqref{eq_bounds_densities} and the same computations as in the proof of \eqref{estim_JfJg_dis}, one can estimate the norm of the additional term so that
\begin{equation}\label{estim_J3_nonlin}
    \|(\rho_{\ggot}^{n+1}-(\rhoinfs)\inv)J_{\fgt}^{n+1}-(\rho_{\fg}^{n+1}-\rhoinfs)J_{\tilde{\ggot}}^{n+1}\|_2\leq\gamma C_{J1}^*\|(I-\Pi^\Delta)\widetilde{\Fg}^{n+1}\|_\Delta.
\end{equation}
Next, looking at the discrete time derivative of $H_\delta^\Delta$, the same computations as in the linearized setting lead to the relation
\begin{equation*}
    H_\delta[\widetilde{\Fg}^{n+1}] - H_\delta[\widetilde{\Fg}^{n}] \leq \frac{1}{2}\left(\|\widetilde{\Fg}^{n+1}\|_\Delta^2 - \|\widetilde{\Fg}^{n}\|_\Delta^2\right) + \delta \sum_{k=1}^4 \tilde{T}_{11k}^n + \delta \tilde{T}_{115}^n + \delta \tilde{T}_{12}^n,
\end{equation*}
where $\tilde{T}_{11k}^n$, $k=1,\dots,4$ and $\tilde{T}_{12}^n$ are defined as $T_{11k}^n$, $k=1,\dots,4$ and $T_{12}^n$ in the proof of Proposition \ref{prop_evol_entropy_dis}, replacing $F$ by $\widetilde{\Fg}$. The additional term $\tilde{T}_{115}^n$ comes from the right-hand side of \eqref{scheme_Jh_nonlin} and is given by
\begin{equation*}
    \tilde{T}_{115}^n = \langle (\rho_{\ggot}^{n+1}-(\rhoinfs)\inv)J_{\fgt}^{n+1}-(\rho_{\fg}^{n+1}-\rhoinfs)J_{\tilde{\ggot}}^{n+1}, D_x^c\Phi^{n+1}\rangle_2.
\end{equation*}
This term can then be estimated using a Cauchy-Schwarz inequality, \eqref{estim_J3_nonlin} and \eqref{estim_dxphi_dis}, producing a cross-term
\begin{equation}\label{estim_T6_nonlin}
    \left|\tilde{T}_{115}^n\right| \leq \gamma C_{J1}^* C_P C_u^* \|(I-\Pi^\Delta)\widetilde{\Fg}^{n+1}\|_\Delta \|\Pi^\Delta\widetilde{\Fg}^{n+1}\|_\Delta.
\end{equation}
Proceeding as in the proof of Proposition \ref{prop_evol_entropy_dis}, choosing an admissible $\delta>0$, we can then establish a similar entropy dissipation as before, but with an additional positive term
\begin{equation}
    H_\delta^\Delta\left[\Fgt^{n+1}\right]-H_\delta^\Delta\left[\Fgt^n\right]\leq -\Delta t\,K_{\delta}\,\left\|\Fgt^{n+1}\right\|_\Delta^2+\Delta t \gamma\left\|\Fgt^{n+1}\right\|_\Delta^2.
\end{equation}
Finally, choosing $\gamma$ small enough, that is such that $\gamma <K_\delta$,
we can conclude the proof as before, using Lemma \ref{lem_eqnorm_dis}.
\end{proof}

\section{Numerical results}\label{sec:numeric}

In this section we shall present numerical results for both the linearized and nonlinear schemes we presented and analyzed in the previous sections. We shall do it for a variety of initial data and equilibrium profiles, in order to showcase both the accuracy and the robustness of our approach. 
Unless stated otherwise, we will always take $v^* = 12$ and $L = 16$ half points in velocity (and have then $32$ control volumes in the velocity space) and $N = 101$ spatial cells for the torus $[0, \pi]$. Due to the unconditional stability of our implicit approach, we shall take the rather large time step $\Delta t = 0.1$ for the linear case. The nonlinear case being more intricated, we shall use an adaptative time stepping to optimize the number of iterations needed in the Newton-Raphson solver for the scheme \eqref{scheme_f_nonlin}--\eqref{scheme_g_nonlin}. Note that due to the relative simplicity of this method, the Jacobian involved in this solver is exact.

\myParagraph{Equilibrium profiles}
We shall consider three different equilibrium densities in our upcoming numerical experiments: a classical centered reduced Gaussian
\begin{equation}
\label{eq:gaussian}
    \chi_\mathcal{M}(v) := \frac1{\sqrt{2 \pi}} \, e^{-| v|^2/2}, \quad \forall v \in [-v^*,v^*];
\end{equation}
a polynomially decaying function 
\begin{equation}
    \label{eq:heavytailed}
    \chi_\mathcal{P}(v) := \frac{1}{1+|v|^4}, \quad \forall v \in [-v^*,v^*];
\end{equation}
a polynomially oscillating function
\begin{equation}
    \label{eq:oscilheavytailed}
    \chi_\mathcal{O}(v) := \frac{\cos(\pi \, v) + 1.1}{1+|v|^6}, \quad \forall v \in [-v^*,v^*].
\end{equation}

\begin{Remark}
    Note that the distribution $\chi_\mathcal{P}$ has only bounded moments up to order $3$, instead of the $4$ needed for Theorem \ref{theo_hypoco_dis_lin} to hold. Nevertheless, we observe in the upcoming numerical experiments that this is not an issue for recovering an exponential decay of our discrete solutions. We infer that this hypothesis is only technical. This is in total agreement with the continuous theorem from \cite{NeumannSchmeiser2016} where this hypothesis is not required.
\end{Remark}

\subsection{Discrete hypocoercivity of the linearized scheme}
\label{sec:subLin}    
    Let us start by investigating the hypocoercive behavior of the linear scheme \eqref{scheme_f_lin}-\eqref{scheme_g_lin}. 
    We shall compare the three fluxes introduced in Section \ref{sec:discrete_setting}, namely the monotone Lax-Friedrichs \eqref{flux_F}-\eqref{flux_G} and upwind \eqref{flux_F_U}-\eqref{flux_G_U}, and the nonmonotone central fluxes \eqref{flux_F_C}-\eqref{flux_G_C}.
    
    \myParagraph{Test 1. Large time behavior with the same equilibrium}
    We first present an experiment where the distributions $\chi_1$ and $\chi_2$ are both fixed as the heavytailed profile $\chi_\mathcal{P}$. We choose as initial data some far from equilibrium distributions given for $x\in[0, \pi]$, $v \in [-v^*,v^*]$ by
    \begin{equation}
        \label{eq:initData}
        \left \{
        \begin{aligned}
            & f_I(x,v) = \exp\left( -((x-\pi/2)^2 + v^2/2)/0.2 \right)/0.1, \\
            & g_I(x,v) = \left(1+\cos(4x)\right) \exp\left( -((x-\pi/2)^2 + v^2/2)/0.2 \right).
        \end{aligned}
        \right.
    \end{equation}

    \begin{figure}
        \centering
        \includegraphics[width=.99\linewidth]{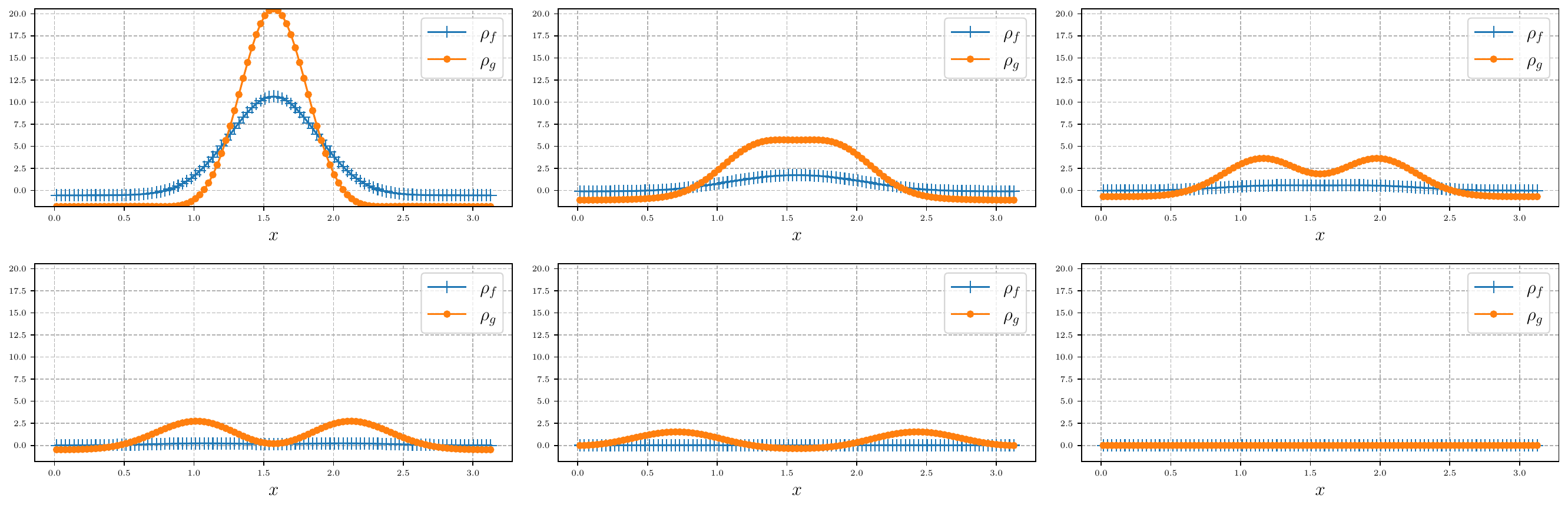}
        \caption{\textbf{Test 1. Large time behavior of the linearized scheme, same equilibria $\chi_\mathcal{P}$.} Snapshots of the densities of the two species at time $t=0$, $0.8$, $1.2$, $1.6$, $2.5$, $50$, using the Lax-Friedrichs fluxes \eqref{flux_F}-\eqref{flux_G}.}
        \label{fig:linPerturbDensities1}
    \end{figure}
     
    We observe in Figure \ref{fig:linPerturbDensities1} that the densities associated with $F$ both converge nicely toward the global equilibrium, without spurious oscillations or loss of the global mass difference.
    
    \begin{figure}
        \centering
        \includegraphics[width=.99\linewidth]{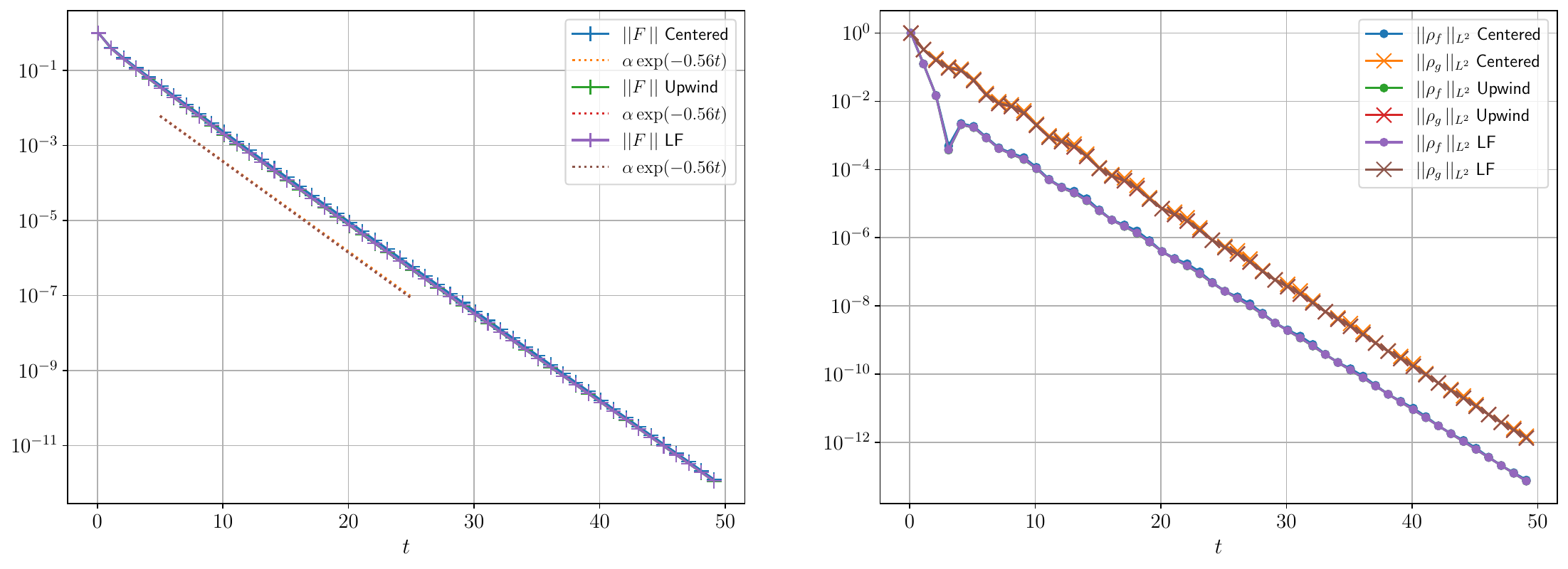}
        \caption{\textbf{Test 1. Large time behavior of the linearized scheme, same equilibria $\chi_\mathcal{P}$.} Time evolution of the weighted $L^2$ norm of the solution $F$ to the linearized problem \eqref{scheme_f_lin}-\eqref{scheme_g_lin} (left) and $L^2$ norm of the densities, for three different fluxes function (right).} 
        \label{fig:trend2EquLin1}
    \end{figure}    
    
    Figure \ref{fig:trend2EquLin1} presents the hypocoercive behavior of our numerical methods. We observe on the left part of the figure that the expected exponential decay of the weighted $L^2$-norm of the solution $F$ to the linear scheme \eqref{scheme_f_lin}-\eqref{scheme_g_lin} holds.
     We notice that the rate of decay is independent on the choice of the numerical flux. This is not unexpected, since monotonicity properties were not crucial in our proof of Theorem \ref{theo_hypoco_dis_lin}. We shall see in the next subsection that this conclusion is not valid in the nonlinear case, where monotonicity was paramount for obtaining the result.
    
    The right part of the figure illustrates the time evolution of the $L^2$-norms of the densities $\rho_f$ and $\rho_g$. These objects are not expected to decay monotonically whatsoever. We nevertheless observe an exponential decay of the upper envelope of these quantities, which expresses the rapid equilibration of the densities toward the same global mass, a consequence of the global preservation of the mass difference.

    \myParagraph{Test 2. Large time behavior with different equilibria}
        Using the same initial datum \eqref{eq:initData} than in the previous test case, we present now simulations of our scheme with $\chi_1 = \chi_\mathcal{P}$ and $\chi_2 = \chi_\mathcal{O}$. 
        The use of the two different equilibrium profiles will enrich the dynamics of this linear problem, showcasing the ability of our approach to deal with multiple species of particles.
        
    \begin{figure}
        \centering     
        \begin{tabular}{c} $f(t,x,v)+\fg_\infty(t,x,v)$ \\
        \includegraphics[width=.97\linewidth]{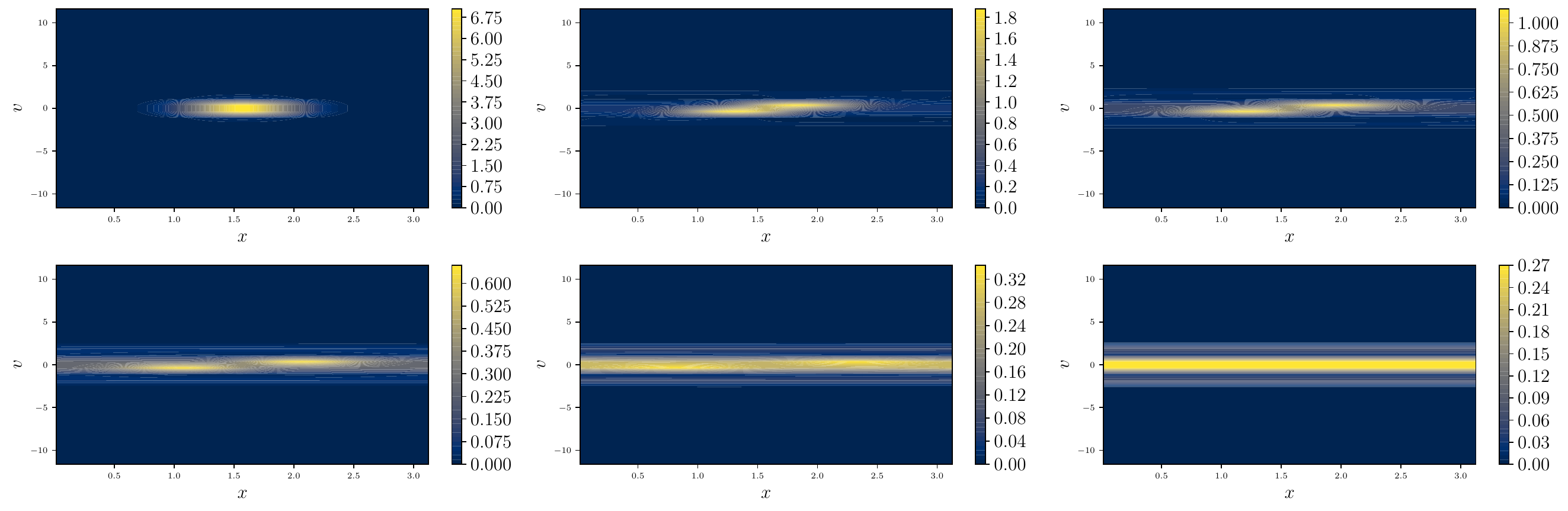} \\
        $g(t,x,v)+\ggot_\infty(t,x,v)$ \\
        \includegraphics[width=.97\linewidth]{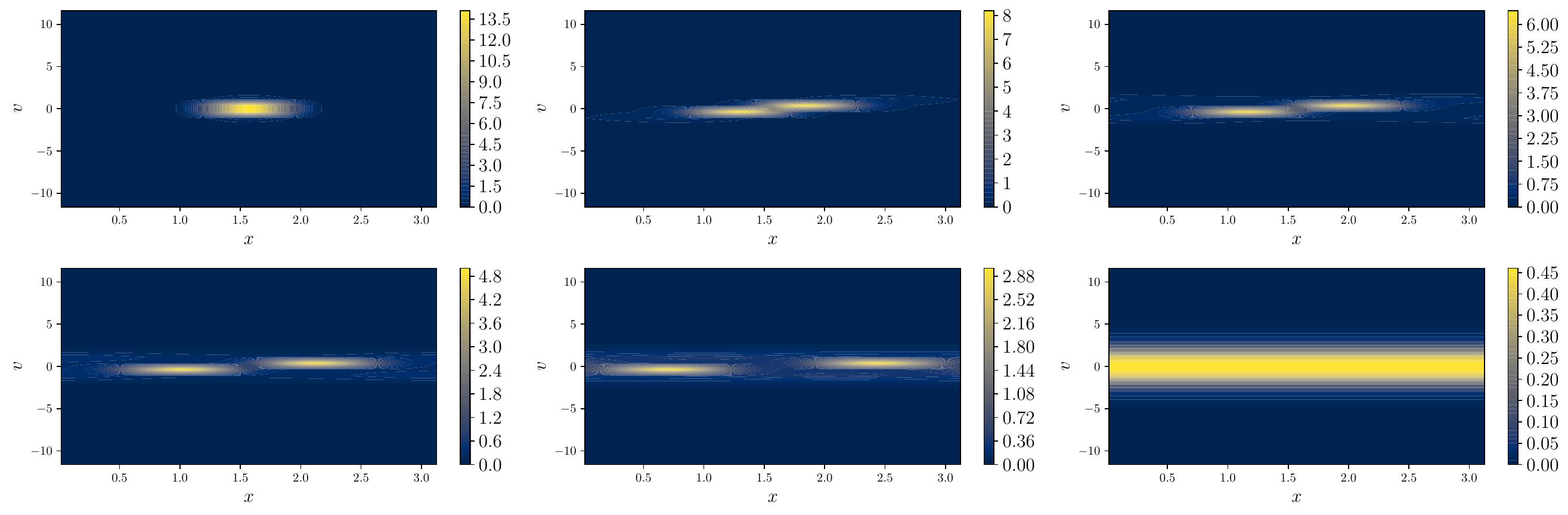}
        \end{tabular}        
        \caption{\textbf{Test 2. Large time behavior of the linearized scheme, different equilibria.} Snapshots of the distribution function of each species at time $t=0$, $0.8$, $1.2$, $1.6$, $2.5$, $50$, using the Lax-Friedrichs fluxes \eqref{flux_F}-\eqref{flux_G}.}
        \label{fig:linPerturbDensities2}
     \end{figure}
 
         Figure \ref{fig:linPerturbDensities2} presents the isosurfaces in the $(x,v)$-plane of the distribution functions $f+\fg_\infty$ and $g+\ggot_\infty$ at different times, using the Lax-Friedrichs flux \eqref{flux_F}-\eqref{flux_G}. We observe the expected relaxation of the initial data of each species towards the correct global equilibrium in large time, without any oscillations or loss of the global mass difference.

    \begin{figure}
        \centering
        \includegraphics[width=.99\linewidth]{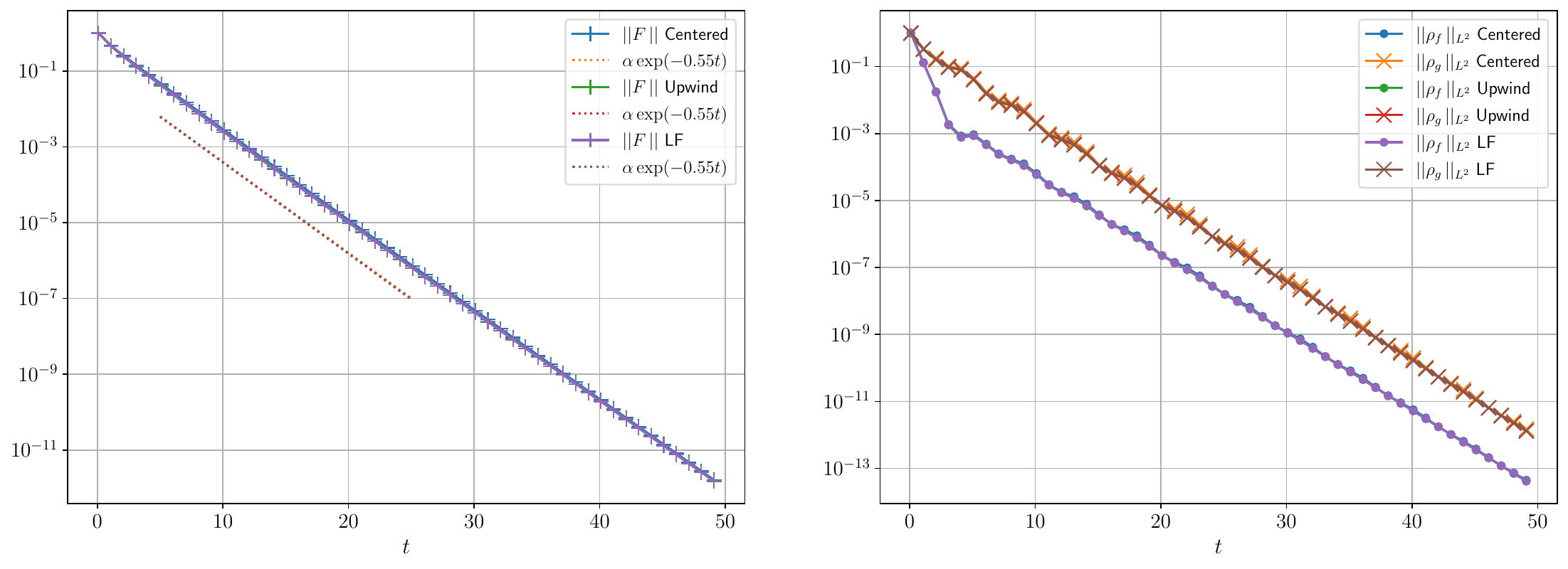}
        \caption{\textbf{Test 2. Large time behavior of the linearized scheme, different equilibria.} Time evolution of the weighted $L^2$ norm of the solution $F$ to the linearized problem \eqref{scheme_f_lin}-\eqref{scheme_g_lin} (left) and $L^2$ norm of the densities, for three different fluxes.}         \label{fig:trend2EquLin2}
    \end{figure}      
    
    The left part of Figure \ref{fig:trend2EquLin2} presents the large time behavior of the weighted $L^2$-norm of $F$, for the three possible choices of flux functions. Once more, we observe the result predicted by Theorem  \ref{theo_hypoco_dis_lin}, namely exponential decay of this quantity towards $0$, even with the heavytailed distribution $\chi_\mathcal{P}$. The choice of the flux (monotone or not) has no impact on this rate of decay.
    The right part of this figure presents the time evolution of the $L^2$-norms of $\rho_f$ and $\rho_g$. We observe once more some small oscillations on these quantities, that are not expected to decay monotonically. Nevertheless, a rapid exponential-like convergence towards $0$ happens also in this more complicated test case.
    
\subsection{Discrete hypocoercivity of the nonlinear scheme}
\label{subsec:DisHypo}
    We now present simulations of the full nonlinear scheme  \eqref{scheme_f_nonlin}-\eqref{scheme_g_nonlin}. The numerical parameters are the same than before, except for the time step, that we choose to be adaptive for the sake of stability of the Newton-Raphson solver. 
    The initial time step is set up at $\Delta t = 110^{-3}$, and it is iteratively doubled according to the behavior of the solver. Indeed, the Newton-Raphson solver used to update the solution to the scheme may break the global conservation of mass difference between the species for a large time  step when the solution is far from equilibrium if its tolerance is too low. Nevertheless, we observe a rapid growth of the time step during the length of the simulation, up to values of order $0.3$. Note that during this refinement process, this value always remains smaller than the critical one needed for  monotonicity of the Lax-Friedrichs flux in Theorem \ref{theo_hypoco_dis_nonlin}, hence enforcing the idea that monotonicity is necessary for discrete hypocoercivity.     
         
    \myParagraph{Test 3. Large time behavior with the same equilibrium}
    
    As a first numerical experiment for this nonlinear model we consider a same Gaussian equilibrium $\chi_\mathcal{M}$ for both species, and a smooth initial data:
    \[
        \left \{\begin{aligned}
            & f_I(x,v) = \chi_\mathcal{M}(v) \, v^4 \left (1+\cos(2 x)\right), \\
            & g_I(x,v) = \left(1+\cos(4x)\right)\, \exp\left( -((x-\pi/2)^2 + v^2/2)/0.2 \right).
        \end{aligned} \right.
    \]
    
    \begin{figure}
        \centering
        \begin{tabular}{c} $\fg(t,x,v)$ \\
            \includegraphics[width=.97\linewidth]{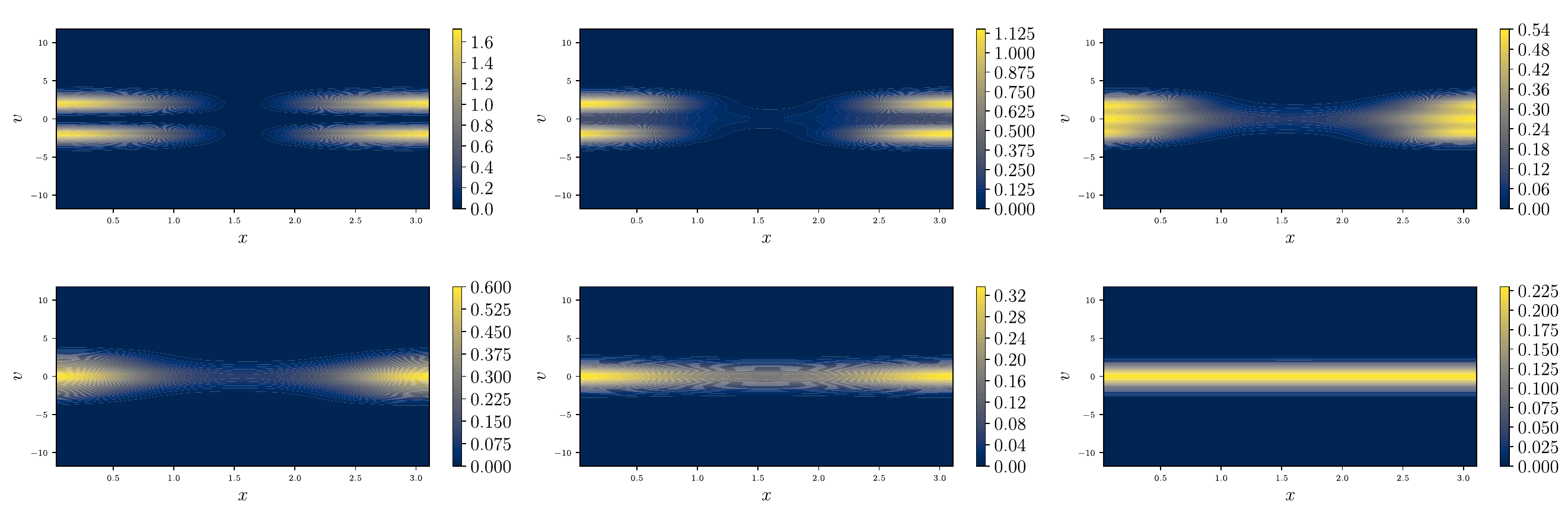} \\
            $\ggot(t,x,v)$\\
            \includegraphics[width=.97\linewidth]{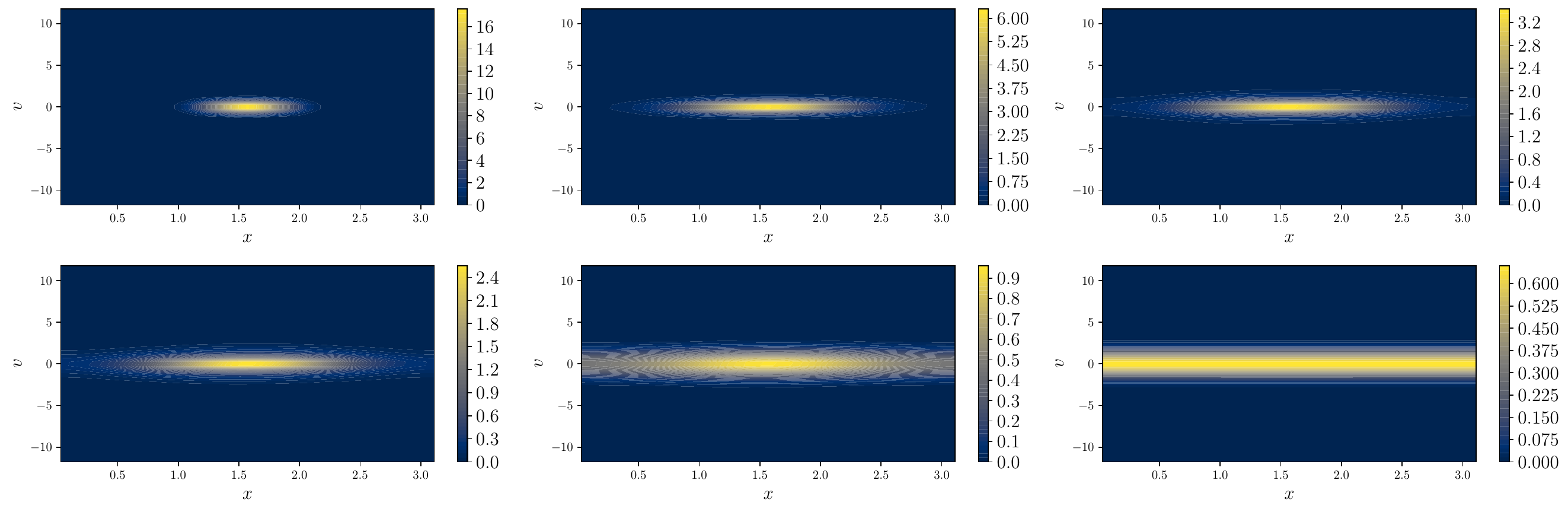}
        \end{tabular}
        \caption{\textbf{Test 3. Large time behavior of the nonlinear scheme.} Snapshots of the distribution function of each species at time $t=0$, $0.83$, $2.25$, $3.35$, $9.67$, $100$, using the Lax-Friedrichs fluxes \eqref{flux_F}-\eqref{flux_G}.}
        \label{fig:NLDistfunctionsSameEq}
    \end{figure} 
    
    We present in Figure \ref{fig:NLDistfunctionsSameEq} some snapshots of the time evolution of the particle distribution function in the $(x,v)$-phase space, for each species $\fg$ and $\ggot$. As in the linear case, we observe a rapid convergence of the far from equilibrium initial data towards the equilibrium distributions $\fg_\infty$ and $\ggot_\infty$, without any spurious oscillations or loss of global mass difference.
    
    \begin{figure}
        \centering
        \includegraphics[width=.99\linewidth]{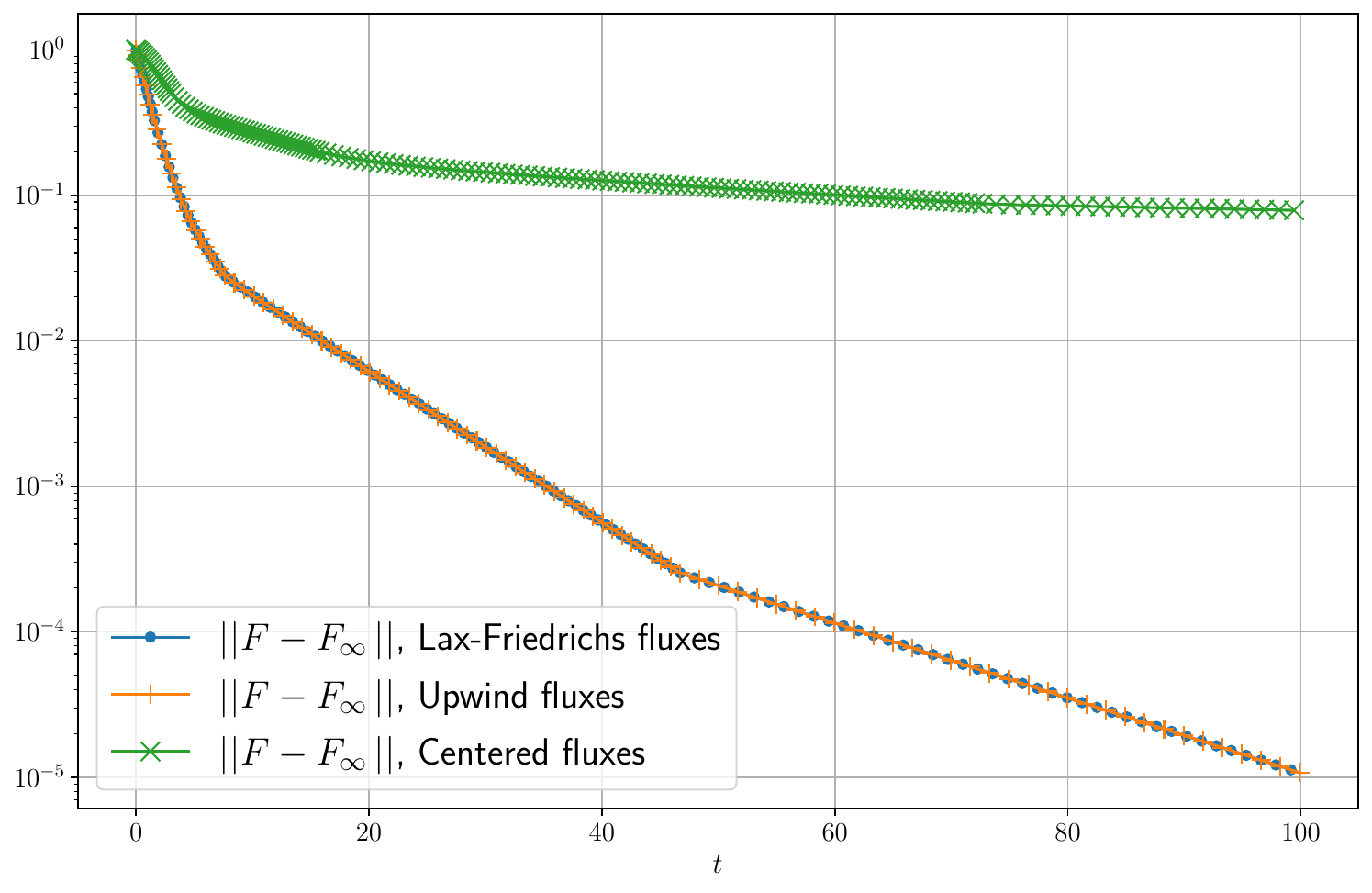}
        \caption{\textbf{Test 3. Large time behavior of the nonlinear scheme.} Trends to equilibrium of $\Fg$ towards $\Fg_\infty$ in the weighted $L^2$-norm, for the monotone fluxes.}
        \label{fig:NLtrend2Equ}
    \end{figure}
    
    Figure \ref{fig:NLtrend2Equ} illustrates the  nonlinear hypocoercivity result established in Theorem \ref{theo_hypoco_dis_nonlin}. It presents the large time behavior of the weighted $L^2$-norm of $\Fg$, for the three possible choices of flux functions. We observe in that particular numerical experiment a huge departure from the simpler linear setting, as well as a confirmation of the fact that monotone fluxes would be an optimal assumption in that result. 
    Indeed, one can see on that figure that the exponential decay \eqref{eq:expDecayNL} of $\|\Fg^n-\Fg^\infty\|_\Delta$ 
    occurs only for the monotone Lax-Friedrichs  \eqref{flux_F}-\eqref{flux_G} and upwind  \eqref{flux_F_U}-\eqref{flux_G_U} fluxes, and not for the nonmonotone centered fluxes. While the former exhibits the expected exponential decay, the latter  will saturate at a value which is dictated by the loss of global mass difference induced by the choice of flux. 
    
    \myParagraph{Test 4. Large time behavior with different equilibria}

        In order to showcase the robustness of our scheme, we  finally choose as an initial condition a fully random uniform initial data for both species. Note here that such initial data does not fulfill the bounds needed in Theorem \ref{theo_hypoco_dis_nonlin} for the result to hold, making them seemingly purely technical. The profiles chosen in this test are the heavytailed  $\chi_1 = \chi_\mathcal{P}$ and $\chi_2 = \chi_\mathcal{O}$. 
    
    \begin{figure}
        \centering
        \begin{tabular}{c} $\fg(t,x,v)$ \\
            \includegraphics[width=.97\linewidth]{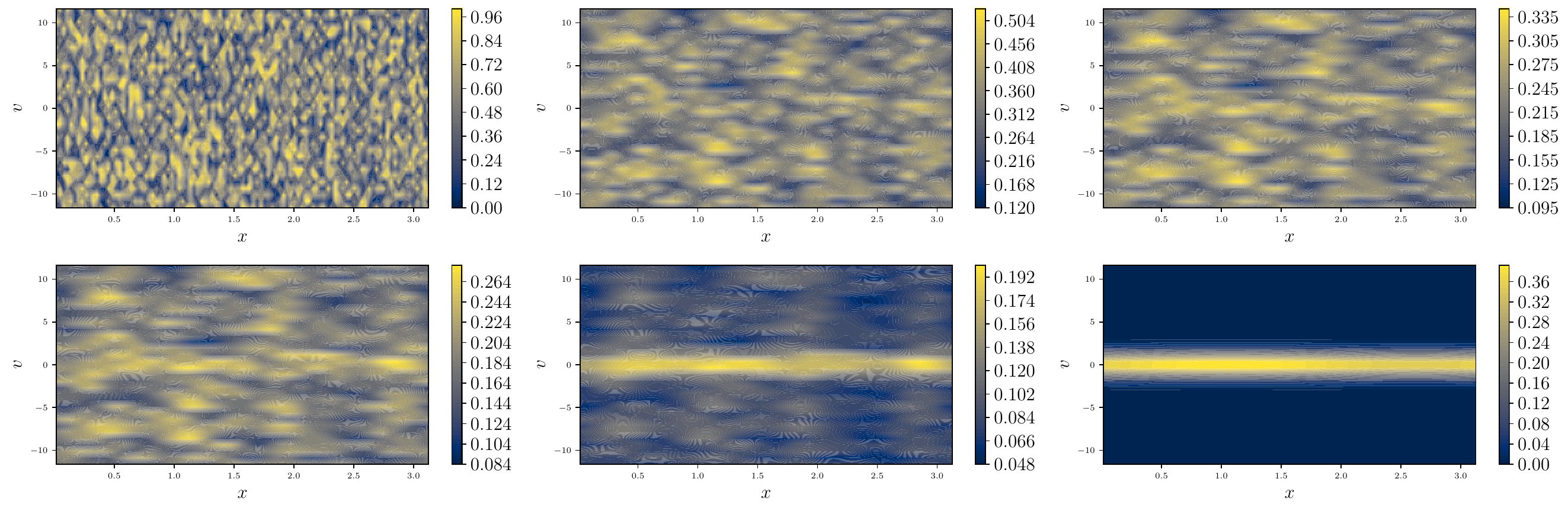} \\
            $\ggot(t,x,v)$\\
            \includegraphics[width=.97\linewidth]{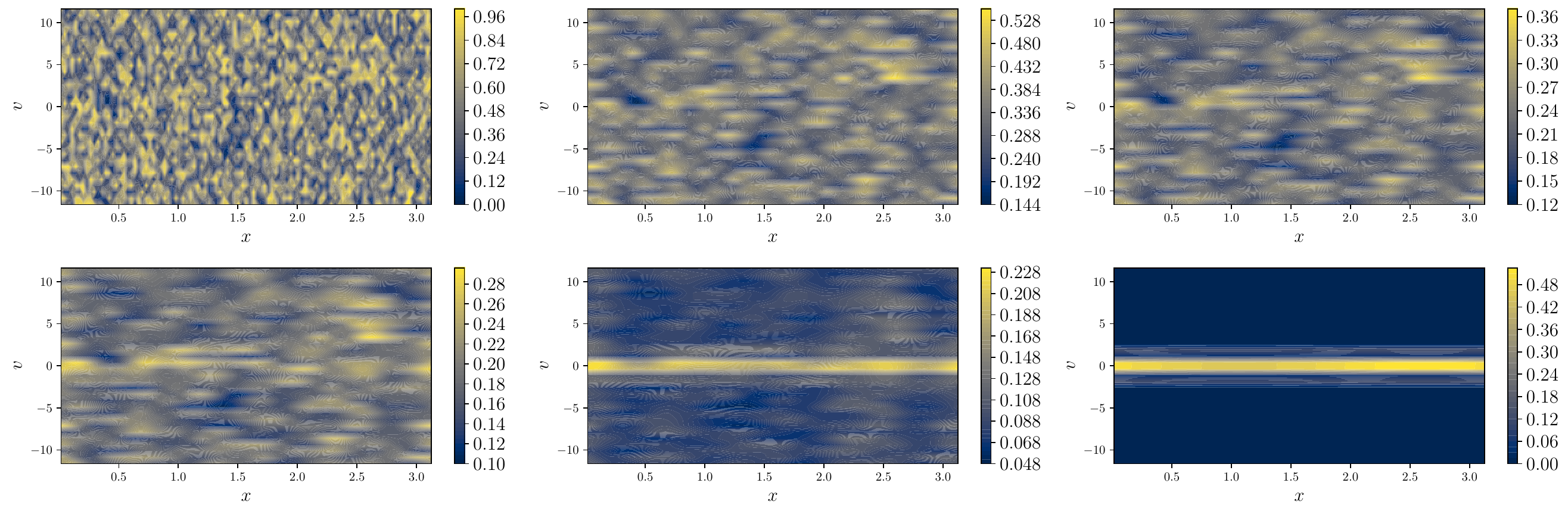}
        \end{tabular}
        \caption{\textbf{Test 4. Large time behavior of the nonlinear scheme.} Snapshots of the distribution function of each species at time $t=0$, $0.16$, $0.38$, $0.77$, $1.66$, $49.9$, using the Lax-Friedrichs fluxes \eqref{flux_F}-\eqref{flux_G}.}
        \label{fig:NLDistfunctionsDiffEq}
    \end{figure} 
    
    We present in Figure \ref{fig:NLDistfunctionsDiffEq} some snapshots of the time evolution of the particle distribution function in the $(x,v)$-phase space, for each species $\fg$ and $\ggot$. We observe again a rapid convergence of the far from equilibrium initial data towards the equilibrium distributions $\fg_\infty$ and $\ggot_\infty$. 
    No spurious oscillations appear with the monotone Lax-Friedrichs fluxes and  the global mass difference is preserved up to machine precision in this test.  
    
    \begin{figure}
        \centering
        \includegraphics[width=.99\linewidth]{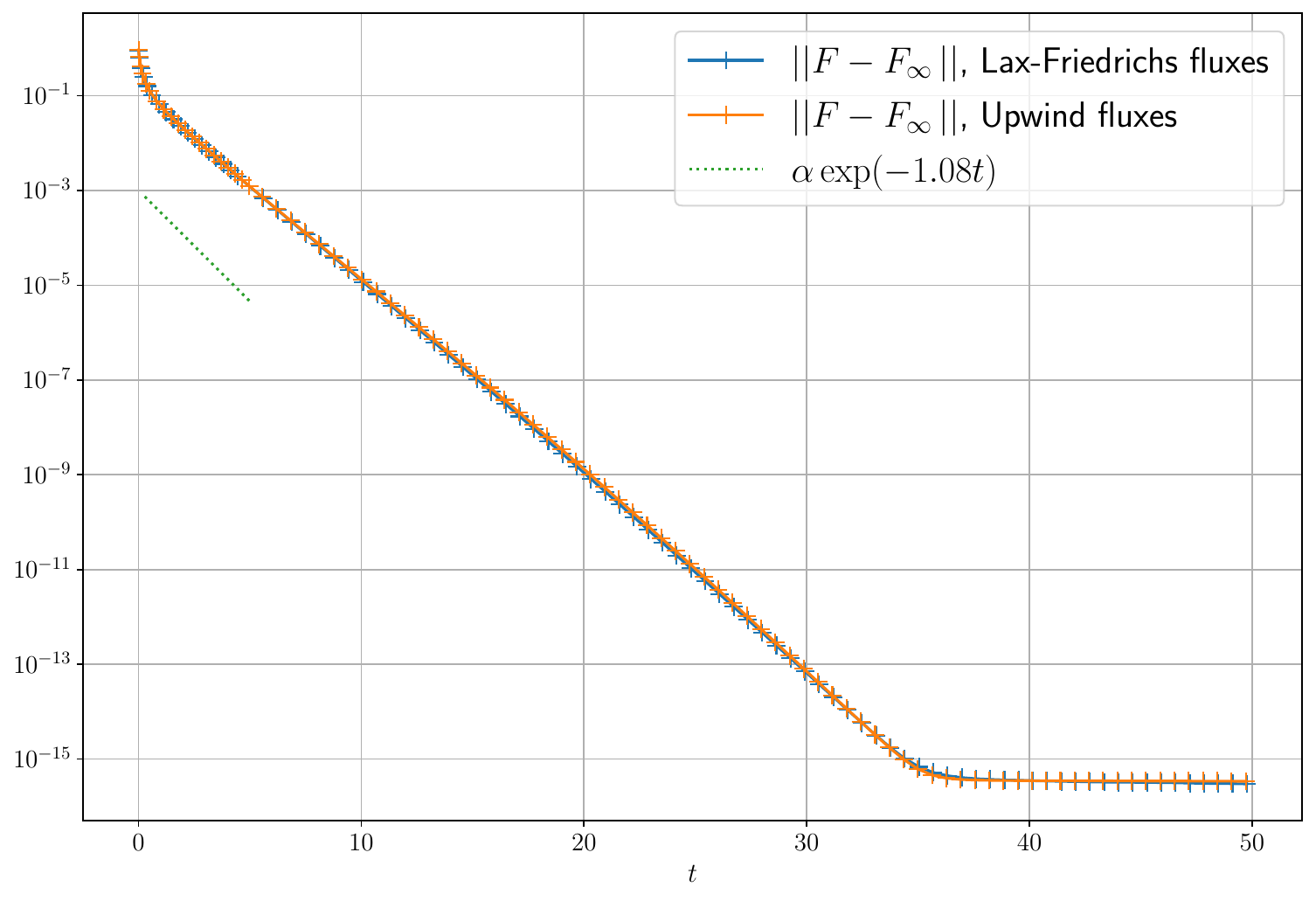}
        \caption{\textbf{Test 4. Large time behavior of the nonlinear scheme.} Trends to equilibrium of $\Fg$ towards $\Fg_\infty$ in the weighted $L^2$-norm, for the monotone fluxes.}
        \label{fig:NLtrend2EquFar}
    \end{figure}
    
    These observations are confirmed in Figure \ref{fig:NLtrend2EquFar}, where we present the time evolution of $\|\Fg^n-\Fg^\infty\|_\Delta$. Note that this far from equilibrium case is not covered by Theorem \ref{theo_hypoco_dis_nonlin}. Nevertheless, this quantity exhibits a fast exponential decay towards $0$, but only for monotone fluxes. We did not chose to present the case of centered fluxes, because we were not able to have a convergent Newton solver in that particular case. This is certainly due to the loss of mass induced by nonmonotone fluxes with such a singular initial data.
    This test case hence emphasizes our claim about the nonoptimality of the bounds needed in Theorem \ref{theo_existence_bornes}.

    \section*{Acknowledgement}
    M. Bessemoulin-Chatard benefits from the support of the French government ``Investissements d’Avenir'' program integrated to France 2030, bearing the following reference ANR-11-LABX-0020-01, and by the ANR Project Muffin (ANR-19-CE46-0004). 
    T. Laidin acknowledges support from Labex CEMPI (ANR-11-LABX-0007-01).
    T. Rey received funding from the European Union's Horizon Europe research and innovation program under the Marie Sklodowska-Curie Doctoral Network Datahyking (Grant No. 101072546), and is supported by the French government, through the UCA$_{JEDI}$ Investments in the Future project managed by the National Research Agency (ANR) with the reference number ANR-15-IDEX-01.

\bibliographystyle{acm}
\bibliography{biblio_hypoco_system}

\end{document}